\documentclass[12pt]{amsart}
\usepackage[top=1in, bottom=1in, left=1in,right=1in]{geometry}
\usepackage[utf8]{inputenc}
\usepackage{amsfonts, amsmath, amssymb, mathrsfs}%, xypic}
\usepackage{amsthm, url}
\usepackage[all]{xy}
\usepackage{graphicx, 
epstopdf}
\usepackage{subfig}
\usepackage{color}
\usepackage{bbm}
\usepackage[draft]{fixme}
\usepackage{bm} 
\usepackage{epigraph}
\usepackage{endnotes}
\usepackage{ifpdf}
\usepackage{todonotes}

\usepackage{multicol}
\usepackage{multirow}
\usepackage{graphicx}
\usepackage{tikz}
\usepackage{comment}
\makeatletter
\newtheorem*{rep@theorem}{\rep@title}
\newcommand{\newreptheorem}[2]{%
\newenvironment{rep#1}[1]{%
 \def\rep@title{#2 \ref{##1}}%
 \begin{rep@theorem}}%
 {\end{rep@theorem}}}
\makeatother

\newtheorem{theorem}{Theorem}[section]
\newreptheorem{theorem}{Theorem}
\newtheorem*{theorem*}{Theorem}
\newtheorem*{lemma*}{Lemma}
\newtheorem{lemma}[theorem]{Lemma}

\newtheorem{corollary}[theorem]{Corollary}
\newtheorem{proposition}[theorem]{Proposition}

\newtheorem{conjecture}[theorem]{Conjecture}

\theoremstyle{definition}

\newtheorem{remark}[theorem]{Remark}
\newtheorem{exampleth}[theorem]{Example}
\newenvironment{example}{\begin{exampleth}}{\hfill $\diamond$  \end{exampleth}}
\usepackage{hyperref}
\usepackage[fit]{truncate}

\usepackage[T1]{fontenc}

\DeclareMathOperator{\Gr}{Gr}
\DeclareMathOperator{\Dr}{Dr}
\DeclareMathOperator{\FlDr}{FlDr}

\DeclareMathOperator{\conv}{conv}

\DeclareMathOperator{\val}{val}

\DeclareMathOperator{\PM}{PM}
\DeclareMathOperator{\cone}{cone}
\DeclareMathOperator{\supp}{supp}

\DeclareMathOperator{\trop}{trop}
\DeclareMathOperator{\tropplus}{trop^+}
\DeclareMathOperator{\tropstar}{trop^*}

\DeclareMathOperator{\nustar}{\nu^*}

\newcommand{\R}{\mathbb{R}}

\newcommand{\N}{\mathbb{N}}
\newcommand{\Z}{\mathbb{Z}}

\newcommand{\C}{\mathbb{C}}

\newcommand{\cR}{\mathcal{R}}
\newcommand{\cC}{\mathcal{C}}
\newcommand{\cK}{\mathcal{K}}

\newcommand\ii{\mathbbm{i}}
\newcommand\kk{\mathbbm{k}}
\newcommand\ps{\{\!\{t\}\!\}}
\newcommand{\Ln}{\mathcal{L}_n}

\begin{document}

\title{Tropicalizing Principal Minors of Positive Definite Matrices}

\author{Abeer Al Ahmadieh}
\address{School of Mathematics, American University of Sharjah, Sharjah, Sharjah, UAE.}
\email{aalahmadieh@aus.edu}

\author{Felipe Rinc\'on}
\address{School of Mathematical Sciences, Queen Mary University of London, London, UK.}
\email{f.rincon@qmul.ac.uk}

\author{Cynthia Vinzant}
\address{School of Mathematics, University of Washington, Seattle, WA, USA.}
\email{vinzant@uw.edu}

\author{Josephine Yu}
\address{School of Mathematics, Georgia Institute of Technology, Atlanta, GA, USA.}
\email{jyu@math.gatech.edu}

\begin{abstract}
We study the tropicalization of the image of the cone of positive definite matrices under the principal minors map. 
It is a polyhedral subset of the set of $M$-concave functions on the discrete $n$-dimensional cube. 
We show it coincides with the intersection of the affine tropical flag variety with the submodular cone. 
In particular, any cell in the regular subdivision of the cube induced by a point in this tropicalization can be subdivided into base polytopes of realizable matroids.  We use this tropicalization as a guide to discover new algebraic inequalities among the principal minors of positive semidefinite matrices of a fixed size. We also extend our results to positive semidefinite matrices via taking closures in the tropical semifield $\mathbb{R}\cup\{-\infty\}$.
\end{abstract}

\maketitle

\section{Introduction}

For an $n\times n$ matrix $A$ and subsets $S, T \subset [n] = \{1,2,\dots,n\}$ of the same size, we will denote by $A(S,T)$ the determinant of the submatrix of $A$ with rows indexed by $S$ and columns indexed by $T$.  
For a square matrix $A$, we use $A_S$ to denote the {\em principal minor} $A(S,S)$.  A real symmetric or Hermitian matrix is called {\em positive definite} if all of its principal minors are positive. 

The {\em tropicalization} $\tropplus(S)$ of a semialgebraic subset $\mathcal{S}\subset \R_+^N$ is the closure of the image of $\mathcal{S}_{\R\ps}$ under coordinate-wise valuation, where $\mathcal{S}_{\R\ps}$ denotes the extension of $\mathcal S$ to the field of real Puiseux series.  It is a rational polyhedral fan and coincides with
the \emph{logarithmic limit set} as in \cite{Ale}:
\[
\tropplus(\mathcal{S}) = \lim_{t \rightarrow \infty} \{(\log_t(x_1), \dots, \log_t(x_N)) : x \in \mathcal{S}\}.
\]  
Tropicalization is a way to track the exponential behavior of $S$ as coordinates go to $0$ or $\infty$.

The principal minors of an $n \times n$ positive definite matrix $A$ can be encoded by the degree-$n$ homogeneous polynomial in $n+1$ variables
\[
f_A = \det({\rm diag}(x_1, \hdots, x_n) + y A)
= \sum_{S \subseteq [n]} A_{S} \cdot {\bf x}^{[n]\backslash S}y^{|S|}  \in \R[x_1,\dots,x_n,y],
\]
where ${\bf x}^T := \prod_{i \in T}x_i$.  This polynomial is {\em stable} \cite{BB06}, which implies that the tropicalization of the set of principal minors of positive definite matrices is a subset of the set of $M^\natural$-concave functions on $\{0,1\}^n$~\cite{Branden2}.
We see in Corollary \ref{cor:equalitydependingonn} that for $n\geq 6$ not every $M^\natural$-concave function arises this way from positive definite matrices.
Our main result is a description of this tropicalization in terms of the affine tropical flag variety. 

\begin{reptheorem}{thm:TropGrassFlagGrassPD}
The tropicalization of the set of principal minors of $n\times n$ positive definite real symmetric (resp.~Hermitian) matrices equals the intersection of the affine tropical flag variety over $\R$ (resp.~$\C$) with the 
cone of submodular functions on the hypercube $\{0,1\}^n$. 
\end{reptheorem}

The study of the tropical flag variety was initiated by Haque in $2012$~\cite{Haque}. Recently, Brandt, Eur, and Zhang  described it in terms of flag matroid polytope subdivisions \cite{BEZ}. 
The totally non-negative part of the tropical flag variety was studied in \cite{Boretsky}.

An $M^\natural$-concave function induces a regular subdivision 
of the unit cube $\{0,1\}^n$ that coarsens a subdivision in which every cell is a {\em dehomogenized matroid polytope}, i.e.\ a 0/1 polytope with edges of the form $e_i-e_j$ or $e_i$ contained in a layer $\{x\in [0,1]^n : k\leq \sum_i x_i \leq k+1\}$ for some $k$. When the $M^\natural$-concave function arises as the tropicalization of the principal minors of a positive definite matrix, we show that these matroid polytopes correspond to realizable matroids.

\begin{reptheorem}{thm:repMatroid}
 The regular subdivision of $[0,1]^n$ induced by the tropicalization of the principal minors of a real symmetric (or Hermitian) positive definite matrix is a coarsening of a subdivision of $[0,1]^n$ into dehomogenized matroid polytopes of matroids realizable over $\R$ (or $\C$, respectively).
\end{reptheorem}

Understanding determinantal inequalities for positive definite matrices is a classical topic of study, see for example \cite{CW,CFB}, and it continues to be a topic of interest \cite{TC,C,JZC,LS20,DWH,BS}. In \cite{TC}, for instance, Hall and Johnson use cone theoretic techniques to characterize the semigroup of ratios of products of principal minors over all positive definite matrices.

Any tropical polynomial inequality valid on the tropicalization of a semialgebraic set $\mathcal S$ can be lifted to a polynomial inequality valid on the set $\mathcal S$ itself \cite{JSY}. We use our description of the tropicalization of the set $\mathcal S$ of principal minors of positive definite matrices as a guide to understanding the possible polynomial inequalities valid on $\mathcal S$; in particular, we study lifts of the $M^\natural$-concavity constraints to polynomial inequalities on $\mathcal S$. 

\begin{example}
Let $\mathrm{PD}_3$ be the set of $3\times 3$ real symmetric positive definite matrices, and let $\mathcal S_3 \subset \R_+^{2^{[3]}}$ be the image of $\mathrm{PD}_3$ under the map sending a matrix $A$ to its vector of principal minors $(A_S)_{S \subset [3]}$. 
Consider the set $\cone(\mathcal S_3) = \{ \lambda\,x \mid \lambda \in \R_+  \text{ and } x \in \mathcal S_3 \}$,
where the coordinate indexed by the subset $\emptyset$ is not required to be equal to $1$. 
Theorem \ref{thm:TropGrassFlagGrassPD} says that $\trop(\cone(\mathcal S_3))$ consists of all vectors $w \in \R^{2^{[3]}}$ that satisfy the submodular inequalities and the tropical incidence relation
\begin{equation}\label{eq:3by3}
    \max(w_{1}+w_{23},w_{2}+w_{13}, w_{3}+w_{12}) \text{ is attained at least twice}.
\end{equation}
Consider now the projection $\mathcal P$ of the semialgebraic set $\cone(\mathcal S_3)$ onto the six coordinates indexed by the subsets $1,2,3,12,13,23$. 
The set $\mathcal P$ is a full-dimensional semialgebraic set of $\R_+^6$; indeed, the six $1 \times 1$ and $2\times 2$ principal minors of a positive definite matrix are algebraically independent.
However, its tropicalization $\trop(\mathcal P)$ satisfies the tropical equation~\eqref{eq:3by3}, and thus it is only a $5$-dimensional polyhedral subset of $\R^6$.
This tropical equation cannot be lifted to an algebraic equation satisfied by $\mathcal P$, but, as we explore in Section \ref{sec:inequalities}, it is instead a consequence of $\mathcal P$ satisfying the following three algebraic inequalities:
\begin{equation*}
    A_1 A_{23} + A_2 A_{13} \geq \frac{1}{2} A_3 A_{12}, \quad A_1 A_{23} + A_3 A_{12} \geq \frac{1}{2} A_2 A_{13}, \quad A_2 A_{13} + A_3 A_{12} \geq \frac{1}{2} A_1 A_{23}.
\end{equation*}
Other inequalities of this form with different coefficients, which also hold for all Lorentzian polynomials, are presented in Theorem \ref{thm:inequ1}.  
\end{example}

\subsection*{Organization} The paper is organized as follows. We introduce relevant background and definitions in Section~\ref{sec:background}, including various characterizations of $M^\natural$-concave functions and their relation to the affine flag Dressian. In Section~\ref{sec:FlagVariety}, we discuss a connection between the tropicalization of the flag variety and a slice of the $(n,2n)$ Grassmannian over arbitrary valuated fields. In Section~\ref{sec:pm}, we specialize these results to real closed and algebraically closed fields of characteristic zero and show that the resulting sets coincide with the tropicalization of the positive definite cone under the principal minor map.
In Section~\ref{sec:PSD} we show that tropicalization of semialgebraic sets is compatible with taking closures and use this to describe the tropicalization of principal minors of positive semidefinite matrices.
In Section~\ref{sec:inequalities}, we explore consequences for polynomial inequalities on the principal minors of positive definite matrices. Finally, in Section~\ref{sec:proof}, we prove a technical lemma used in Section~\ref{sec:FlagVariety}.

\subsection*{Acknowledgements}
We are grateful to Jonathan Boretsky, H.\ Tracy Hall,  Yassine El Maazouz, and Bernd Sturmfels for helpful discussions and insights.  CV is partially supported by the NSF DMS grant \#2153746. JY is partially supported by the NSF DMS grants \#1855726 and \#2348701.

\section{Definitions and Background}\label{sec:background}

\subsection{Tropicalization}\label{subsec:prelim}

A (nonarchimedean) {\em valuation} on a field $\cK$ is a map $\val : \cK^* \rightarrow \R \cup \{\infty\}$ satisfying
\begin{align*}\val(ab) &= \val(a)+\val(b),\\ \val(a+b)  &\geq \min(\val(a),\val(b)), \\
 \val(a) = \infty &\iff a = 0.
\end{align*}
The image $\Gamma = \val(\cK^*)$ of $\val$ is an additive subgroup of $\R$ called the \emph{value group}, where we denote $\cK^* = \cK \setminus \{0\}$. The valuation is \emph{non-trivial} if $\Gamma \neq \{0\}$.  
In order to use the {\bf max convention}, which is more compatible with inequalities over real fields, we will take $\nu:\cK \to \R\cup\{-\infty\}$ to be $\nu(a) = -\val(a)$ so that 
\begin{align*}\nu(ab) &= \nu(a)+\nu(b)\\
	\nu(a+b) & \leq \max(\nu(a),\nu(b))\\
     \nu(a) = -\infty &\iff a = 0.
\end{align*}
We use $\kk$ to denote the \emph{residue field} of $\cK$, 
which is the quotient of the valuation ring $\mathcal{O} = \{a\in \cK : \val(a)\geq 0\}$ by its unique maximal ideal $\mathfrak{m} = \{a\in \cK : \val(a)> 0\}$. 

Throughout the paper, we use  $\cK$ to denote a field with a nontrivial valuation with an infinite residue field $\kk$. The valuation has a {\em splitting} or a {\em cross-section} if there is a group homomorphism $\phi$ from the value group $\Gamma$ to the multiplicative group $\cK^*$ such that $\val \circ\,\phi$ is the identity on $\Gamma$.
If $\cK$ is real closed or algebraically closed, then there is a splitting by \cite[Lemma 2.4]{AGS} and \cite[Lemma 2.1.15]{MaclaganSturmfels}.
The splitting $\phi$ allows us to talk about the \emph{leading coefficient} of an element $a\in \cK^*$ as the element $\overline{t^{-\val(a)}a}$ in the residue field, where $ t = \phi(1)$.  Typical examples include the fields of rational functions, Laurent series, or Puiseux series over a field $\kk$.  The Puiseux series field $\kk\ps$ is real closed if $\kk$ is real closed and is algebraically closed if $\kk$ is algebraically closed of characteristic $0$.

Given a subset $\mathcal{S}\subset \cK^n$, we define 
\begin{align*}
\nu(\mathcal{S}) &= \{(\nu(x_1), \hdots, \nu(x_n)) : x\in \mathcal{S}\} \subset (\Gamma \cup \{-\infty\})^n \\
\nustar(\mathcal{S}) &= \{(\nu(x_1), \hdots, \nu(x_n)) : x\in \mathcal{S} \cap (\cK^*)^n\}
\subset \Gamma^n.
\end{align*}
For $S\subset (\cK^*)^n$, we have $\nustar(S) = \nu(S)$, so we may use these two notations interchangeably in this case.

For an algebraic subset $\mathcal{S} \subset \C^n$, defined by some polynomial equations over $\C$, and for an algebraically closed field extension $\cK \supset \C$ with nontrivial valuation, let $S_\cK$ be the algebraic subset of $\cK^n$ defined by the same polynomial equations.
Then the {\em tropicalization} of $\mathcal{S}$ is defined as
\[
\tropstar(S) = \overline{\nustar(S_\cK)}.
\]
It follows from the Fundamental Theorem of Tropical Geometry~\cite{MaclaganSturmfels} that this does not depend on the choice of the extension $\cK$.

Let $\cR$ be a real closed field with a nontrivial valuation.  Any real closed field has a unique total ordering $\geq$ compatible with the field operations, where $a \geq b$ if $a-b$ is a complete square.  We will always assume that the valuation is compatible with the order, namely,
 if $a,b\in \cR$ with $0 < a < b$, then $\nu(a) \leq \nu(b)$.
An example is the field of real Puiseux series $\R\ps$, where a series is positive if its leading coefficient is positive.

For a semialgebraic subset $\mathcal{S} \subset \R^n$ and a real closed field extension $\cR \supset \R$, let $\mathcal{S}_\cR$ be the semialgebraic subset of $\cR^n$ defined by the same semialgebraic expression (a first-order formula in the language of ordered rings) defining $\mathcal{S}$.  By the Tarski--Seidenberg transfer principle, the set $S_\cR$ does not depend on the choice of semialgebraic description, as the set of points satisfying one semialgebraic description but not the other is empty over $\R$ so it is empty over $\cR$ as well.

The {\em tropicalization} of $\mathcal{S} \subset \R^n$ is defined to be closure of the image of $S_\cR$ under coordinate-wise valuation:
\[\tropstar(\mathcal{S}) = \overline{\nustar(\mathcal{S}_\cR)}\]
where the closure is taken in the Euclidean topology on $\R^n$.

This set does not depend on the choice of the extension $\cR$ or the choice of semialgebraic expression defining $\mathcal{S}$~\cite{JSY}, and it coincides with the logarithmic set if $\mathcal{S}$ is defined over $\R$~\cite{Ale}. 
Moreover we have $\tropstar(\mathcal{S})\cap \Gamma^n = \nustar(\mathcal{S}_\cR)$ \cite[Theorem\  4.2]{Ale}, and $\tropstar(\mathcal{S})$ is a union of polyhedra~\cite[Theorem\ 3.1]{AGS}.

Finally, we use $\cC$ to denote the degree-two field extension $\cC=\cR(\ii)$ where $\ii^2=-1$. Since $\cR$ is real-closed, $\cC$ is algebraically closed. This comes with complex conjugation  $\overline{a+\ii b} = a-\ii b$ where $a,b\in \cR$. The valuation on $\cR$ naturally extends to $\cC$ by $\nu(a+\ii b) = \max\{\nu(a), \nu(b)\}$.

If $X\subset (\C^*)^n$ is an algebraic variety, then its image under coordinate-wise absolute-value is a semialgebraic subset $|X|$ of $\R_+^n$.
The tropicalization of $X$ can be defined as the tropicalization of $|X|$. 
We can similarly define $X_\cC$ to be the algebraic variety of $\cC^n$ defined by the vanishing of all polynomials that vanish on $X$. 
A part of the fundamental theorem of tropical geometry states that the logarithmic limit set of $|X|$ coincides with the closure of the image of $X_\cC$ under coordinate-wise valuation. That is, $\trop(X) = \overline{\nu(X)}$. A proof can be found, for example, using \cite[Theorem~2]{Bergman} and \cite[Proposition~3.8]{Gubler}.

The principal minors of a matrix give a map from the set of $n\times n$ matrices with entries in $\cR$ to $\cR^{2^{[n]}}$, sending $A \mapsto (A_S)_{S \subset [n]}$. The image of the sets of symmetric and Hermitian matrices have been studied in \cite{HS,oeding,LS09,AV1,AV2} and it has many applications in probability and statistics, see for instance \cite{Yassine}. A symmetric matrix  over a real field $\cR$, or a Hermitian matrix over $\cC$, is called {\bf positive definite} or PD  if all of its principal minors are  positive.  Let $\PM^+_n(\cR)$ and $\PM^+_n(\cC) \subset \cR^{2^{[n]}}$ denote the image of the symmetric and Hermitian PD cone respectively under this principal minor map. Because the image of the principal minors of positive definite matrices is a semialgebraic set, the theorems above imply that 
\[
\trop(\PM^+_n(\R)) = \overline{\nu(\PM^+_n(\cR))} \text{ and }
\trop(\PM^+_n(\C)) = \overline{\nu(\PM^+_n(\cC))}.
\]
\subsection{Discrete Concavity}  
For a set $S\subset [n]$ and an element $i\in [n]$, let us use the shorthand $Si$ for $S\cup \{i\}$ and $S\backslash i$ for the set difference $S\backslash \{i\}$.
A function $F: 2^{[n]} \rightarrow \R \cup \{-\infty\}$ is {\bf submodular} if for all $S, T \subset [n]$,
\[ F(S\cap T) + F(S\cup T) \leq F(S) + F(T).\]
If $F$ is submodular then its {\bf support}
\[ \supp(F) = \{ S \subset [n] \mid F(S) \neq -\infty\} \]
is a {\bf convex collection of sets}, meaning 
\begin{center}
    whenever $S \subset X \subset T$ and $S,T \in \supp(F)$, we have $X \in \supp(F)$.
\end{center} 
In fact, submodular functions can be ``locally'' characterized as follows.

\begin{proposition}
    A function $F: 2^{[n]} \rightarrow \R \cup \{-\infty\}$ is submodular if and only if its support is a convex collection of sets and for all $S\subset[n]$ and distinct $i,j \in [n]\setminus S$,
\[ F(S) + F(Sij) \leq F(Si) + F(Sj).\] 
\end{proposition}
\begin{proof}
These ``local'' inequalities are just the submodular inequalities obtained in the case the sets $S$ and $T$ have the same size and differ by only one element. Moreover, if $F$ is submodular, whenever $S \subset X \subset T$ and $S,T \in \supp(F)$, the submodular inequality $F(S) + F(T) \leq F(X) + F(S \cup (T\setminus X))$ implies that $X \in \supp(F)$, showing that $\supp(F)$ is a convex collection of sets.

For the converse, suppose $F: 2^{[n]} \rightarrow \R \cup \{-\infty\}$ has convex support and satisfies all the ``local'' submodular inequalities. Take $S, T \subset [n]$, and denote $\{s_1, s_2, \dots, s_k\} := S \setminus T$ and $\{t_1, t_2, \dots, t_l\} := T \setminus S$. 
We want to show that $F(S\cap T) + F(S\cup T) \leq F(S) + F(T)$. We can assume that $S\cap T, S \cup T \in \supp(F)$, as otherwise the inequality trivially holds.
For $0 \leq i \leq k$ and $0 \leq j \leq l$, denote $X_{i,j} = (S\cap T) \cup \{s_1, \dots ,s_i\} \cup \{t_1, \dots, t_j\}$, which is in $\supp(F)$ since this is a convex collection of sets.
For $1 \leq i \leq k$ and $1 \leq j \leq l$, we then have the local inequality 
$$F(X_{i-1,j-1}) + F(X_{i,j}) - F(X_{i,j-1}) -F(X_{i-1,j}) \leq 0.$$
Adding these inequalities over all values of $i$ and $j$ and canceling terms, we obtain
$F(X_{0,0}) + F(X_{k,l}) - F(X_{k,0}) -F(X_{0,l}) \leq 0$,
which is the desired inequality.
\end{proof}

%The {\em alcoved triangulation} of the cube $[0,1]^n$ is the triangulation into $n!$ simplices where each simplex is the convex hull of an edge path from $(0,0,\dots,0)$ to $(1,1,\dots,1)$ whose sum of coordinates is monotonously increasing along the path.  Equivalently, it is given by dicing the cube with the type-$A$ braid arrangement hyperplanes $x_i = x_j$.  This triangulation is ``dual'' to the permutohedron.
%Lov\'asz showed that a function of full support $2^{[n]}$ is submodular if and only if the regular subdivision of the unit cube induced by the {\em lower} convex hull is a coarsening of the alcoved triangulation~\cite{Lov}.  In particular, the edges in the lower convex hull correspond to pairs of subsets $S,T \subset[n]$ with $S \subset T$.  Thus the edges with directions $e_i - e_j$ cannot appear in the lower hull and can only appear in the upper hull. {\color{blue} This last discussion is for full support}

The notions of $M$-convex and $M^\natural$-convex functions were introduced by Murota and collaborators~\cite{Murota03}. $M$-concave functions are also called {\em valuated polymatroids}.
A function $F: 2^{[n]} \rightarrow \R \cup \{-\infty\}$ is an {\bf $M^\natural$-concave function} if for all $S, T \subset [n]$ and all $i \in S\backslash T$, either 
\begin{itemize}
	\item $F(S)+F(T) \leq F(S\backslash i) + F(Ti)$, or 
 \item there exists $j\in T\backslash S$ such that 
	$F(S)+F(T) \leq F(Sj\backslash i) + F(Ti\backslash j)$. 
\end{itemize}
Note that the support of an $M^\natural$-concave function is an {\bf $M^\natural$-convex collection of sets}, meaning
for all $S, T \in \supp(F)$ and all $i \in S\backslash T$, either 
\begin{itemize}
	\item $S\backslash i$ and $Ti$ are in $\supp(F)$, or 
 \item there exists $j\in T\backslash S$ such that 
	$Sj\backslash i$ and $Ti\backslash j$ are in $\supp(F)$. 
\end{itemize}
Every $M^\natural$-convex collection of sets is a convex collection of sets. Indeed, when $T \subset S$, the second case in the definition of an $M^\natural$-convex collection of sets is impossible, and thus we can reach any subset $X$ satisfying $T \subset X \subset S$ by repeatedly applying the first case for various elements $i \in S\setminus T$. However, not every convex collection of sets is $M^\natural$-convex, as exemplified, for instance, by the collection $\{\{1,2\},\{3,4\}\}$.
We will see some characterizations of $M^\natural$-concave functions below in Proposition~\ref{prop:Mconcave}.

Our interest in $M^\natural$-concave functions comes from the following fact.
\begin{proposition}\label{prop:TropPMconcavity}
	The tropicalization of the image of the positive definite cone under the principal minor map is a subset of the set of $M^\natural$-concave functions on $\{0,1\}^n$.  
\end{proposition}
\begin{proof}
	If $A \in \R\ps^{n \times n}$ is a positive definite matrix then, using \cite{BB06}, the polynomial
	$f= \det({\rm diag}(x_1,\hdots,x_n) + A) = \sum_{S\subseteq [n]} A_S {\bf x}^{[n]\backslash S}$ is stable and has positive coefficients, so $(\nu(A_S))_{S\subseteq [n]}$ is an $M^\natural$-concave function by \cite{Branden2}.
\end{proof}

\subsection{\texorpdfstring{$M^\natural$}--concavity and valuated matroids}
In this subsection we compile various characterizations of $M^\natural$-concave functions in terms of valuated matroids and polytope subdivisions.

For integers $1 \leq k \leq n-1$, the (affine) {\bf  Dressian} $\Dr(k,n)$ is the polyhedral subset of $(\R \cup \{-\infty\})^{\binom{n}{k}}$ consisting of all functions $p : \binom{[n]}{k} \to \R \cup \{-\infty\}$ such that for any subsets $S \in \binom{[n]}{k-1}$ and $T \in \binom{[n]}{k+1}$, 
\[
\max_{i \in T \setminus S} \, (p(Si)+p(T\backslash i)) \quad \text{is attained at least twice}.
\]
Points in the Dressian $\Dr(k,n)$ are, up to scaling, in one-to-one correspondence with 
rank-$k$ valuated matroids on the ground set $[n]$ or equivalently, $k$-dimensional tropical linear spaces in $\R^n$.

The {\bf affine flag Dressian} $\FlDr(n)$ is the polyhedral subset of $(\R \cup \{-\infty\})^{2^{[n]}}$ consisting of all functions $p : 2^{[n]} \to \R \cup \{-\infty\}$ such that for any subsets $S,T \subset [n]$ with $|S| \leq  |T| - 2$, 
\[
\max_{i \in T \setminus S} (p(Si) + p(T\backslash i)) \quad \text{is attained at least twice}.
\]
When $|S| = |T| - 2$ the conditions above are simply the Pl\"ucker relations on the points $(p(S))_{S \in \binom{[n]}{k}}$ with $k = |S|+1$, while when $|S| < |T| - 2$ they are incidence relations among the corresponding tropical linear spaces. Flag Dressians have been studied in \cite{BEZ}.

For a function $F : 2^{[n]} \to \R \cup \{-\infty\}$, its {\bf multisymmetric lift} is the function
$\hat F : \binom{[2n]}{n} \to \R \cup \{-\infty\}$ given by 
$$\hat F (T) := F(T \cap [n]).$$   
The homogenization of the $k$ to $k+1$ layer is the function $\tilde F_k: \binom{n+1}{k+1} \to \R \cup \{-\infty\}$ given by
  $$\tilde F_k (S) := \begin{cases}
  F(S) & \text{if } n+1 \notin S \\
  F(S\backslash\!(n+1)) & \text{if } n+1 \in S.
  \end{cases}$$

A function $F: 2^{[n]} \rightarrow \R \cup \{-\infty\}$ induces a {\em regular subdivision} of the $0/1$ polytope 
\[\textstyle P_{\supp(F)} = \conv \{e_S \in \{0,1\}^n \mid S \in \supp(F)\}, \quad \text{where } e_S = \sum_{i\in S} e_i,\]
by lifting in a new dimension each of the vertices $e_S$ of $P_{\supp(F)}$ to height $F(S)$, and then projecting back to $\R^n$ the upper convex hull of the lift. 

% swapping (3) and (4) from the arXiv version
\begin{proposition}
	\label{prop:Mconcave}
	For a function $F : 2^{[n]} \rightarrow \R \cup \{-\infty\}$, the following are equivalent.
	\begin{enumerate}
		
  \item The function $F$ is $M^\natural$-concave.

  \item The multisymmetric lift $\hat F$ belongs to the Dressian $\Dr(n,2n)$.

  \item The collection $\supp(F)$ is $M^\natural$-convex, $F$ is submodular, and it belongs to the affine flag Dressian $\FlDr(n)$.
  
  \item The collection $\supp(F)$ is $M^\natural$-convex, $F$ is submodular, and for every $1 \leq k \leq n-2$, the homogenized layer~$\tilde F_k$ belongs to the Dressian $\Dr(k+1,n+1)$.

\item The function $F$ induces a regular subdivision of the polytope $P_{\supp(F)}$ (via upper hull) in which every edge has the form $e_i - e_j$ or $e_i$.
	\end{enumerate}
\end{proposition}

\begin{proof}
The equivalence $(1) \iff (2)$ follows from \cite[Proposition 1.4]{GRSU}.

The equivalence $(1) \iff (3)$ can be found, for instance, in \cite[Theorem 3.2]{murota2016} and \cite[Theorem 10]{RGP}.
%, which says that $F$ is $M^\natural$-concave
%if and only if $F$ is submodular and
% for all $S\subset[n]$ and $i,j,k,l \in [n] \setminus S$, each of the maxima
%	\begin{align*}
%    &\max\{F(Sij) + F(Sk) ,  F(S ik) + F(Sj),   F(Sjk) + F(Si) \}\\
%    &\max\{F(Sij) + F(Sk\ell) ,  F(Sik) + F(Sj\ell),   F(Sjk) + F(Si\ell) \}
%	\end{align*}
%	is attained at least twice. 

The equivalence $(3)\iff (4)$ follows from \cite[Theorem~5.1.2]{BEZ}.
 
The equivalence $(1) \iff (5)$ holds more generally for functions on subsets of $\Z^n$.  By
  \cite[Theorem 6.30]{Murota03}, a function is $M$-concave on a finite subset of $\Z^n$ with a fixed coordinate sum if and only if every cell in the regular subdivision induced by the function via the upper hull is an $M$-convex set. By \cite[Theorem 4.15]{Murota03} and \cite[Lemma 2.3]{MUWY},  $M$-convex sets are integer points in generalized permutohedra, which are polyhedra with edges in directions $e_i - e_j$. 
Projecting out one of the coordinates gives the desired result for $M^\natural$-concave functions.
\end{proof}

\begin{example}\label{ex:3x3}
    Consider the matrix
    \[
    A = \begin{pmatrix}
        1 & 1 & 1 \\ 1 & 1+t^4 & 1+t^3 \\ 1 & 1 + t^3 & 1 + t^2 + t^4
    \end{pmatrix} =B^T B \ \ \text{ where } \ \ B = \begin{pmatrix}
        1 & 1 & 1 \\ 0 & t^2 & t \\ 0 & 0 & t^2
    \end{pmatrix}.
    \]
    The matrix $A$ is positive definite since $A = B^T B$ with $\det(B)\neq 0$. For $w_S = -\val(A_S)$, we have $w_{\emptyset} = w_1=w_2=w_3 = 0$, $w_{12} =-4$, $w_{13}=w_{23} = -2$, and $w_{123} = -8$. The function $S\mapsto w_S$ is $M^\natural$-concave and induces the regular subdivision of the cube shown in Figure~\ref{fig:cube}.\end{example}

     \begin{figure}    \begin{center}
		\begin{tikzpicture}
                \fill[red!20!white] (0,0,0)--(0,2,0)--(2,2,2)--(2,0,2)--(0,0,0);
                \fill[red!20!white] (2,0,0)--(0,2,0)--(2,2,2)--(2,0,0);
                \fill[red!20!white] (0,0,0)--(2,0,2)--(0,2,2)--(0,0,0);
			\draw[black,line width=2pt] (0,2,0) -- (2,2,0) -- (2,2,2) -- (0,2,2) -- (0,2,0) --cycle;
			\draw[black,line width=2pt] (0,2,2) -- (2,2,2) -- (2,0,2) -- (0,0,2) -- (0,2,2) --cycle;
			\draw[black,line width=2pt] (2,2,2) -- (2,2,0) -- (2,0,0) -- (2,0,2) -- (2,2,2) --cycle;
			\draw[black, line width=2pt,dashed] (0,0,0)--(0,0,2);
			\draw[black, line width=2pt,dashed] (0,0,0)--(2,0,0);
			\draw[red, line width=2pt,dashed] (0,0,0)--(0,2,0) ;
                \draw[red, line width=2pt,dashed] (0,0,0)--(0,2,2) ;
                 \draw[red, line width=2pt,dashed] (2,0,0)--(0,2,0) ;
                 \draw[red, line width=2pt] (2,0,0)--(2,2,2) ;
                \draw[red, line width=2pt] (2,0,2)--(0,2,2) ;
			\draw[red, line width=2pt] (0,2,0)--(2,2,2) ;
			\draw[red, line width=2pt] (2,2,2)--(2,0,2) ;
			\draw[red, line width=2pt,dashed] (0,0,0)--(2,0,2) ;
			\node[scale=0.7] at (-0.2,0,2.5) {$w_{0}=$ \textcolor{red}{$0$}};
			\node[scale=0.7] at (2.5,0,2.5) {$w_{1}=$ \textcolor{red}{$0$}};
			\node[scale=0.7] at (-0.4,2,2.5) {$w_{3}=$ \textcolor{red}{$0$}};
			\node[scale=0.7] at (3,2.2,2.6) {$w_{13}=$ \textcolor{red}{$-2$}};
			\node[scale=0.7] at (-0.3,-0.2,0) {$w_{2}=$ \textcolor{red}{$0$}};
			\node[scale=0.7] at (2.7,0,0) {$w_{12}=$ \textcolor{red}{$-4$}};
			\node[scale=0.7] at (0,2.2,0) {$w_{23}=$ \textcolor{red}{$-2$}};
			\node[scale=0.7] at (2.4,2.2,0){$w_{123}=$ \textcolor{red}{$-8$}};
		\end{tikzpicture}
   \end{center}
 \caption{The regular subdivision induced by the matrix in Example~\ref{ex:3x3}.\\
 }\label{fig:cube}
    \end{figure}
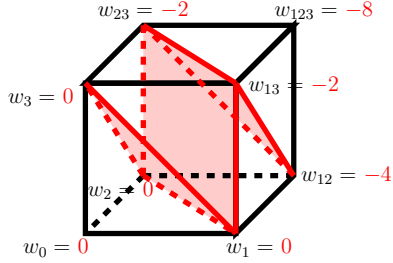

We call a function $F:2^{[n]}\to \R \cup \{-\infty\}$ {\bf strictly submodular} if  
\[F(S) + F(Sij) < F(Si) +F(Sj) \quad \text{whenever } F(S) + F(Sij) \neq -\infty\]
for all $S\subset [n]$ and distinct $i,j\in [n]\backslash S$.

\begin{lemma}\label{lem:layers}
Let $F:2^{[n]}\to \R \cup \{-\infty\}$ be $M^{\natural}$-concave. The function $F$ is strictly submodular if and only if each cell of the regular subdivision of $P_{\supp(F)}$ induced by $F$ via upper hull is contained in $\{x \in [0,1]^n : k\leq\sum_{i=1}^nx_i\leq k+1\}$ for some $k=0,\hdots, n-1$.
\end{lemma}

\begin{proof}
    ($\Leftarrow$) Suppose that every cell of the subdivision induced by $F$ has the desired form. Consider $S\subset[n]$ with $|S|=k \leq n-2$ and $i,j\in [n]\backslash S$ such that $e_S, e_{Sij} \in P_{\supp(F)}$. 
    By assumption, no cell of the subdivision can contain both points $e_{S}$ and $e_{Sij}$. Therefore, on the square with vertices $e_{S}, e_{Si}, e_{Sj}, e_{Sij}$, $F$ must induce the subdivision with an edge between $e_{Si}$ and $ e_{Sj}$. It follows that $F(Si)+F(Sj)>F(S)+F(Sij)$.

    ($\Rightarrow$) Suppose that $F$ is strictly submodular and suppose, for the sake of contradiction, that a cell $C$ of the subdivision induced by $F$ is not contained in a slice of the cube with the desired form. Let $S\subset [n]$ be of minimum size $|S|=k$ with $e_S\in C$. Let $|x|$ denote $\sum_i x_i$. By assumption, $\max\{|x|: x\in C\}\geq k+2$. 
    Since $e_S$ is a vertex of $C$, there is some edge 
    $[e_S, e_S+v]$ of $C$
    for which $|e_S+v|>|e_S|$. By $M^{\natural}$-concavity, this edge direction must be parallel to $e_i$ for some $i\in [n]\backslash S$. Then $e_{Si}$ is also a vertex of $C$. Since this does not achieve the maximum of $|x|$ over $C$, by the same reasoning, there exists some $j\in [n]\backslash(Si)$ for which $e_{Sij}\in C$. 
    By $M^{\natural}$-concavity, $[e_S, e_{Sij}]$ cannot be an edge of $C$, implying that $e_{Sj}\in C$ as well. 
    The only way for these four points to belong to the same cell $C$ is for $F(Si)+F(Sj) = F(S)+F(Sij)$, contradicting the strict submodularity of $F$.
\end{proof}

\begin{remark}
    The assumption of $M^{\natural}$-concavity is necessary in Lemma~\ref{lem:layers}. For example, consider the function $F:2^{[4]}\to \R$ defined by 
$F(\emptyset)=F([4])=-6$, $F(i)= -3$ for $i\neq 1$ and $F(1)=0$, $F(S)=-1$ for $|S|=2$, and $F(S)=-3$ for $|S|=3$ with $S\neq 234$ and $F(234)=0$. One can check that $F$ is strictly submodular but that the edge $[e_1, e_{234}]$ appears in the induced subdivision via the upper hull. 
\end{remark}

\section{Tropical Grassmannians and  Tropical Flag Varieties}\label{sec:FlagVariety}

In this section  we will show a new relationship between tropical Grassmannians and tropical flag varieties, analogous to the equivalence $(2)\iff(3)$ in Proposition~\ref{prop:Mconcave}.   

The (complete) flag variety is a subvariety of a product of projective spaces $\prod_{k=1}^n \mathbb{P}^{\binom{n}{k}-1}$. 
    An element $(p_k)_{k \in [n]}$ belongs to the flag variety if there is a complete flag of linear spaces $\{0\} = L_0 \subset L_1\subset L_2\subset \hdots \subset L_n = \cK^n$ so that $p_k$ is the vector of Pl\"ucker coordinates of $L_k$.  
    The (complete) {\bf affine flag variety}  is the variety 
    \({\rm Fl}_{\cK}(n) \subset \cK^{2^n} = \prod_{k=0}^n K^{\binom{n}{k}}\) consisting of vectors $(q_S)_{S\subseteq [n]}$ which, when considered as a sequence $(q_\varnothing) , (q_S : |S|=1) , (q_S : |S|=2) , \dots , (q_S : |S|=n-1) , (q_{[n]})$, belong to the flag variety. Each component $(q_S : |S|=k)$ can be scaled independently by nonzero elements of $\cK$.  In particular we allow $(q_\varnothing)$ and $(q_{[n]})$ to take arbitrary values in $\cK^*$.

Consider the Pl\"ucker embedding of the partial flag variety 
  \[\{(L_k, L_{k+1})\in \Gr_{\cK}(k,n)\times \Gr_{\cK}(k+1,n): L_k\subset L_{k+1}\}. 
 \] 
 Its image under the homogenizing map taking $q_S$ to $q_{S(n+1)}$ if $|S|=k$ and to $q_S$ if $|S|=k+1$ coincides with the Pl\"ucker embedding of  $\Gr_{\cK}(k+1,n+1)$. 
 An analogous homogenization appeared in Proposition~\ref{prop:Mconcave}(3). 
 Concretely, if $L_k$ and $L_{k+1}$ are the span of the top $k$ and $k+1$ rows of a $(k+1)\times n$ matrix $M$, then the corresponding  subspace in $\Gr_{\cK}(k+1,n+1)$ is obtained as the row span of the $(k+1)\times (n+1)$ matrix  obtained by appending the column $e_{k+1}$ to $M$. 

 Let $\nu({\rm Fl}_{\cK}(n)) \subset \Gamma^{2^{[n]}} \subset \R^{2^{[n]}}$ be the $\nu$ values of the points in the affine flag variety  ${\rm Fl}_{\cK}(n)$, none of whose Pl\"ucker coordinates are zero.  This set has lineality space containing the tropical scaling of each factor; that is, for any $w:2^{[n]}\to \R$ in $\nu({\rm Fl}_{\cK}(n))$ and any vector of scalars $\lambda = (\lambda_0, \hdots, \lambda_n)\in \Gamma^{n+1}$,
 the function $\lambda\cdot w: 2^{[n]}\to \R$ given by
\begin{equation}\label{eq:tropScaling}
    (\lambda\cdot w)(S) =  \lambda_{|S|} + w(S)
\end{equation}
 also belongs to $\nu({\rm Fl}_{\cK}(n))$.
The following lemma says that the tropical flag variety $\nu({\rm Fl}_{\cK}(n))$ can be recovered from its intersection with the submodular cone using lineality.

\begin{lemma}
\label{lem:submodular}
For any $w:2^{[n]}\to \R$, there are tropical scalars $\lambda = (\lambda_0, \hdots, \lambda_n)\in \Gamma^{n+1}$ so that $\lambda\cdot w$ is strictly submodular. Moreover, if $w$ is already submodular and $\Gamma$ is dense in $\R$, then $\lambda$  can be chosen to be arbitrarily small. 
\end{lemma} 

\begin{proof}
 Let $\lambda_0=\lambda_1=0$ and for each $k=2, \hdots, n$, inductively choose $\lambda_k$ so that 
$$\lambda_k < \min_{\stackrel{S \in {\binom{[n]}{k}}}{i,j \in S}} \ w(S\backslash i) + w(S\backslash j) +2\lambda_{k-1} - w(S\backslash \{i,j\})- w(S)  - \lambda_{k-2},$$
which implies that $\lambda \cdot w$ is strictly submodular.

If $w$ is already submodular, then $ w(S\backslash i) + w(S\backslash j)  - w(S\backslash \{i,j\})- w(S)\geq 0$. We can choose $\lambda_0=\lambda_1=0$ and $\lambda_2 = -\varepsilon$ where $\varepsilon > 0$ is arbitrarily small. For $k \geq 3$ we can choose $\lambda_k$ to satisfy $3\lambda_{k-1} <\lambda_k<\min\{0,2\lambda_{k-1}  - \lambda_{k-2}\}$.
Inductively, this gives $|\lambda_k|< 3^{k-1} \varepsilon$, thus $\lambda$ can be arbitrarily small.
\end{proof}

Our main result in this section, stated below, says that the realizable cases of the equivalence $(2)\iff(3)$ in Proposition~\ref{prop:Mconcave} agree.  
In the next section we will use this to describe the tropicalization of the set of principal minors of positive definite matrices, when $\cK$ is either $\cR$ or $\cC$, as a linear slice of the tropical Grassmannian $\trop({\rm Gr}_{\cK}(n, 2n))$ and as the intersection of the tropical affine flag variety $\trop({\rm Fl}_{\cK}(n))$ with the submodular cone.  However, the result below holds for more general fields $\cK$ and may be of independent interest.

\begin{theorem}\label{thm:TropGrassFlagGrass_AnyField} 
Let $\cK$ be a field with a nonarchimedean valuation with a splitting, value group $\Gamma$, and an infinite residue field.
For a function $F: 2^{[n]}\rightarrow \Gamma$, the following are equivalent:
   \begin{enumerate}
    \item The multisymmetric lift $\hat{F}$ is the valuation of a point in the Grassmannian ${\rm Gr}_{\cK}(n, 2n)$.
    \item The function $F$ is submodular and is the valuation of a point in the affine flag variety ${\rm Fl}_{\cK}(n)$.
    \end{enumerate}   
%    
%    For any positive integer $n$ we have
%	\[
%    \nu({\rm Gr}_{\cK}(n, 2n)) \cap \Ln
%    \ \  = \ \
%   \nu({\rm Fl}_{\cK}(n)) \cap {\rm SUBMOD}_n.
%    \]
%    If $\cK$ is algebraically closed or real closed, we thus have
%    	\begin{equation}\label{eq:mainset}
%    \trop({\rm Gr}_{\cK}(n, 2n)) \cap \Ln
%    \ \  = \ \
%   \trop({\rm Fl}_{\cK}(n)) \cap {\rm SUBMOD}_n.
%    \end{equation}
\end{theorem}

\begin{proof}[Proof of Theorem~\ref{thm:TropGrassFlagGrass_AnyField}($(1)\implies(2)$)]
 Since both %$\nu({\rm Gr}_{\cK}(n, 2n)) \cap \Ln$ and $\nu({\rm Fl}_{\cK}(n)) \cap {\rm SUBMOD}_n$ 
 statements are invariant under global tropical scaling (i.e.\ adding a constant function), we may assume that the function $F: 2^{[n]} \to \Gamma$ satisfies $F(\varnothing)=0$.
Suppose $M \in \cK^{n\times 2n}$ is such that the valuation of its $n\times n$ minors satisfy $\nu(M([n],\tilde T)) = \hat F(\tilde T) = F(\tilde T \cap [n])$ for all $\tilde T \in \binom{[2n]}{n}$. The submatrix $M([n],\{n+1,\dots,2n\})$ has valuation $0$ by our assumption, so we can assume that $M$ has the form $\begin{pmatrix}B & I\end{pmatrix}$ for some $n\times n$ matrix ${B}$ and the $n \times n$ identity matrix $I$. Since the valuation of an $n\times n$ minor of $M$ with columns indexed by $\tilde T$ depends only on the intersection $\tilde T \cap [n]$, the valuation of any minor ${B}(T,S)$ is determined by $S$ independently of $T$.
Consider the flag $L_1\subset L_2\subset \dots \subset L_n = \cK^n$ where $L_k$ is the span of the first $k$ rows of ${B}$. Its Pl\"ucker coordinates are given by the minors $({B}(\{1,\hdots, k\}, S): S\in \binom{[n]}{k}) = (F(S): S\in \binom{[n]}{k})$, which shows that $F$ is the valuation of a point in ${\rm Fl}_{\cK}(n)$.
%, which are equal to the corresponding $n\times n$ minors of $M$.     
%This shows that the valuation of the $n\times n$ minors of $M$ belongs to $\nu({\rm Fl}_{\cK}(n))$. 
Moreover, since ${\rm Gr}_{\cK}(n, 2n)$ is contained in the Dressian
${\rm Dr}(n, 2n)$, 
submodularity follows from the $(2)\implies(3)$ part of Proposition~\ref{prop:Mconcave}.
\end{proof}

For the reverse implication, we need the following technical lemmas.  Recall that $\nu$ is the negative of the valuation.

\begin{lemma}
    \label{lem:TechnicalLemma}
    Let $B\in \cK^{n\times n}$ be an upper triangular matrix 
    such that the function $2^{[n]} \rightarrow \R$ given by  $S\mapsto \nu(B(\{1,\hdots, |S|\}, S))$ is submodular. Then 
	\begin{equation}\label{eq:topJustifiedIneq2}
		\nu(B(\{1,\hdots, |S|\}, S)) \geq \nu(B(T, S))
	\end{equation}
	for all $S,T \subset [n]$ with $|S|=|T|$. 
\end{lemma}
We defer the proof of Lemma~\ref{lem:TechnicalLemma} to Section~\ref{sec:proof}. We will say that $B$ is {\bf top heavy} if it is upper triangular and satisfies the condition \eqref{eq:topJustifiedIneq2}.

\begin{lemma}\label{lem:genCoordChange}
Let $\cK$ be a valued field with a splitting and an infinite residue field.
	Let  $B, C\in \cK^{n\times n}$, $\tilde{B} = CB$ and suppose $C$ is generic with entry-wise valuation zero. Then the valuation of a minor of $\tilde{B}$ depends only on the choice of the columns, not on the choice of the rows.
More precisely, 
	for any subsets $S, T\subseteq [n]$ of size $|S|=|T|=k$, 
	\[\nu(\tilde{B}(T,S)) = \max_{T'\in \binom{[n]}{k}} \nu(B(T',S)).\]
In particular, this is independent of $T$.

If, in addition, $B$ is a top heavy matrix as defined by \eqref{eq:topJustifiedIneq2} and  $C$ is generic lower triangular with valuation zero for the nonzero entries, then 
	\[
	\nu(\tilde{B}(T, S))  =
	\nu(B([k], S))
	\]
	for all subsets $S, T\subseteq [n]$ of the same size $k$. 
\end{lemma}

Here ``generic'' means that the leading coefficients of the entries of $C$ lie in a nonempty Zariski open set over the residue field.  We are defining the leading coefficient using the splitting.  If $\cK$ is a Puiseux series field over a ground field $\kk$, then we can take $C$ to be a generic matrix over $\kk$.   

\begin{proof}
  Consider two $k\times k$ submatrices of $\tilde{B}$ on the same set of columns.  We can transform one to the other by swapping two rows at a time, so let us assume that these two submatrices differ only at one row.  Each row of $\tilde{B}$ is a generic linear combination of rows of $B$.  By linearity of determinant, each of the two minors is a generic linear combination of the same $n$ minors where the row being swapped is replaced by each of the original rows.  By genericity there is no cancellation of leading terms so the linear combinations must have the same valuation, as desired.

  Now suppose $B$ is top heavy and $C$ is lower triangular.  Then each row of $\tilde{B}$ is a generic linear combination of a top-justified subset of rows of $B$.  When we repeatedly expand a $k \times k$ minor of $\tilde{B}$ using the linearity of determinants, we see that its valuation agrees with the valuation of the top justified minor by the top-heaviness of $B$.
\end{proof}

\begin{example}[Example~\ref{ex:3x3} continued] 
    Consider the matrix $B$ from Example~\ref{ex:3x3}.
    One can check that the function $S\mapsto \nu(B([|S|], S)$ on $\{0,1\}^3$ is submodular. Since $B$ is also upper triangular, it follows from Lemma~\ref{lem:TechnicalLemma} that $B$ is top heavy. 
    Consider a ``generic'' lower triangular matrix 
    \[C = \begin{pmatrix} 1 & 0 & 0 \\ 1 & 1 & 0 \\ 2 & 1 & 1\end{pmatrix} \ \ \text{ with } \ \ \tilde{B} = CB = \begin{pmatrix} 1 & 1 & 1 \\ 1 & 1+t^2 & 1+t \\ 2 & 2+t^2 &2+t+t^2 \end{pmatrix}. \]
    Then by the Cauchy-Binet identity,
\begin{align*}\tilde{B}(23, 23) & = C(23,12)B(12,23)+C(23,13)B(13,23)+C(23,23)B(23,23)\\
    & =  (-1)(t-t^2) +(1)(t^2) + (1)(t^4) = -t+2t^2+t^4.
    \end{align*}
In particular, $\nu(\tilde{B}(23, 23))=-1$ is equal to the maximum $\nu$ value of a minor $B(ij,23)$, which is achieved by $ij=12$, since $B$ is top heavy.  Similarly, one can check that $\nu(\tilde{B}(ij, 23))=-1$ for every pair $ij$. 
\end{example}

We can now finish the proof of Theorem~\ref{thm:TropGrassFlagGrass_AnyField}.

\begin{proof}[Proof of Theorem~\ref{thm:TropGrassFlagGrass_AnyField}($(2) \implies (1)$)]  
	Since both of these conditions are invariant under global tropical scaling (adding a constant function), we may assume that $F(\varnothing)=0$.
    Suppose that $F:2^{[n]}\to \Gamma$ satisfies condition $(2)$. 
    %belongs to ${\nu}({\rm Fl}_{\cK}(n)) \cap {\rm SUBMOD}_n$.  We will show that it belongs to $\nu({\rm Gr}_{\cK}(n, 2n)) \cap \Ln$. 
	Then
 \[F(S) = \nu(B([k],S)) - \lambda_k,\]
	where $B$ is an upper triangular $n\times n$ matrix over $\cK$ and $\lambda_1, \hdots \lambda_n$ are in the value group $\Gamma$.
 The rows $v_1, \hdots, v_n$ of $B$ represent the flag 
 \[
 {\rm span}\{v_1\} \subset {\rm span}\{v_1,v_2\} \subset \dots \subset {\rm span}\{v_1,\hdots,v_n\} = \cK^n
 \]
in the flag variety and  $\lambda_1, \hdots \lambda_n$ results from scaling of the Pl\"ucker coordinates of each linear space in the flag.	
    Since $F(S)$ is finite for every $S$, the minors $B([k], [k])$ are 
    nonsingular. 
 	After rescaling the rows
 $v_1 \rightarrow t^{\lambda_1} v_1$, $v_2 \rightarrow t^{\lambda_2-\lambda_1} v_2$, $v_3 \rightarrow t^{\lambda_3-\lambda_2} v_3,\dots, v_n \rightarrow t^{\lambda_n-\lambda_{n-1}} v_n$ we may assume that 
 \[F(S) = \nu(B([k],S)).\]
 By Lemma~\ref{lem:TechnicalLemma}, $B$ is top heavy.
Let $\tilde{B} = LB$ where $L$ is a generic lower triangular matrix as in Lemma~\ref{lem:genCoordChange}. So for any $S\subset [n]$ with $|S|=k$
 \[F(S) = \nu(B([k],S)) = \nu(\tilde{B}(T,S))\]
 for all $T \in {\binom{[n]}{k}}$ by Lemma~\ref{lem:genCoordChange}.  
%Since $F(\varnothing) = 0$, the vector $w$ coincides with the $\nu$ values of the $n \times n$ minors of the matrix $(\tilde{B}~I)$ under the identification of $\R^{2^{[n]}}$ with the linear space $\Ln$ as in \eqref{eqn:Ldef}, and so
This shows that the multisymmetric lift $\hat F$ is the valuation of the $n \times n$ minors of the matrix $(\tilde{B}~I)$, thus $F$ satisfies condition $(1)$.
\end{proof}

\begin{remark}
El Maazouz \cite{Yassine} defines the {\em entropy} map associated to a lattice $\Lambda = B\mathcal{O}^n\subset \cK^n$ to be the function on $\{0,1\}^n$ defined by $I\mapsto \min_{|J|=|I|}\val(\det(B(I,J)))$, which coincides with the formula in Lemma~\ref{lem:genCoordChange} above up to sign and transpose. Interestingly, in \cite{Yassine}, this function is primarily considered over fields $\cK$ with finite residue field, for which Lemma~\ref{lem:genCoordChange} does not apply. For fixed $n$, one should be able to amend the arguments above to apply when the residue field $\kk$ is finite but sufficiently large. 
\end{remark}

\section{Tropicalizing principal minors of positive definite matrices}
\label{sec:pm}

In this section we show that
the tropicalization of the 
    image of the positive definite cone under the principal minor map coincides with the subsets of tropical Grassmannians and tropical flag varieties studied in the previous section.  As before we will use $\cR$ and $\cC$ to denote a real closed field and an algebraically closed field respectively, with nontrivial nonarchimedean valuation, with value group $\Gamma$.

Consider the linear subspace $\Ln$ of $\mathbb R^{\binom{[2n]}{n}}$ containing all multisymmetric lifts of functions in $\mathbb R^{2^{[n]}}$, namely
\begin{equation}
\label{eqn:Ldef}
\Ln = \left\{ p \in \mathbb R^{\binom{[2n]}{n}} : p(C) = p(D)  \text{ whenever } C\cap [n] = D\cap [n]\right\}.    
\end{equation}
The linear subspace $\Ln$ is isomorphic to $\R^{2^{[n]}}$ via the map
\begin{equation*}
\label{eq:id}   
 \phi: \Ln \xrightarrow[]{\ \ \cong \ \ }  \mathbb R^{2^{[n]}}
\end{equation*}
sending any $p \in \Ln$ to $\phi(p) \in \R^{2^{[n]}}$ given by $\phi(p)(S) = p(S \cup T)$ with $T$ any subset of $\{n+1,\dots ,2n\}$ of size $n-|S|$.

The equivalence $(2)\iff(3)$ in Proposition~\ref{prop:Mconcave} can be rephrased as saying
\[
    {\rm Dr}(n, 2n) \cap \Ln
    \ \  = \ \
   {\rm FlDr}(n) \cap {\rm SUBMOD}_n,
\]
where ${\rm SUBMOD}_n$ denotes the cone of submodular functions on $\{0,1\}^n$ and we identify points in $\Ln$ with points in $\R^{2^{[n]}}$ as above.  On the other hand, in Lemma~\ref{lem:genCoordChange} says that the negative valuation of the $n\times n$ minors of the matrix $\begin{pmatrix}\tilde{B} & I_n \end{pmatrix}$ belongs to 
	$\nu({\rm Gr}_{\cK}(n, 2n)) \cap \Ln$.

We will now state our main result, which describes the tropicalization of the space of principal minors of positive definite matrices in terms of some of the subsets studied in previous sections.
%compares various representable analogues of the constructions in Proposition~\ref{prop:Mconcave}.  
\begin{theorem}\label{thm:TropGrassFlagGrassPD} 
	Fix a positive integer $n$ and a field $\cK = \cR$ or $\cC$.  A function $F: 2^{[n]}\rightarrow \Gamma$ is the valuation of the principal minors of a positive definite matrix over~$\cK$, up to global tropical scaling, if and only if it satisfies any of the two equivalent conditions in Theorem~\ref{thm:TropGrassFlagGrass_AnyField}. 
    %, the following are equivalent.
    %\begin{enumerate}
    %\item The function $F$ is the valuation of the principal minors of a positive definite matrix over~$\cK$, up to global tropical scaling.
    %\item The multisymmetric lift $\hat{F}$ is the valuation of a point in the Grassmannian ${\rm Gr}_{\cK}(n, 2n)$.
    %\item The function $F$ is submodular and is the valuation of a point in the affine flag variety ${\rm Fl}_{\cK}(n)$.
    %\end{enumerate}
In other words, we have
	\[
    \nu({\rm cone}(\PM^+_n(\cK)))
    \ \ = \ \ 
    \nu({\rm Gr}_{\cK}(n, 2n)) \cap \Ln
    \ \  = \ \
    \nu({\rm Fl}_{\cK}(n)) \cap {\rm SUBMOD}_n,
    \]
where ${\rm cone}(\PM^+_n(\cK)) = \{\lambda a : \lambda\in \cK, a\in \PM^+_n(\cK)\}$.
We take the cone to allow the principal minor corresponding to the empty set take values other than $1$.
Taking euclidean closure in the equation above gives
	\begin{equation}\label{eq:mainset}
    \trop({\rm cone}(\PM^+_n(\cK)))
    \ \ = \ \ 
    \trop({\rm Gr}_{\cK}(n, 2n)) \cap \Ln
    \ \  = \ \
    {\rm trop}({\rm Fl}_{\cK}(n)) \cap {\rm SUBMOD}_n.
    \end{equation}
\end{theorem}
We do not know whether the representable analogue of (4) in Proposition~\ref{prop:Mconcave}, with the Dressians replaced by tropical Grassmannians, is also equivalent to the other sets in this theorem.

\begin{proof}
Let $\mathcal{K} =\cR$ or $\cC$, and let $\kk$ denote its residue field.
The equality of the last two sets is the content of Theorem~\ref{thm:TropGrassFlagGrass_AnyField}.  We will now show that the first two sets are equal. 

    ($\subseteq$) Let $A$ be a positive definite matrix over $\cK$.
    Then $A = B^*B$ for some $B\in \cK^{n\times n}$. 
By Cauchy-Binet, \[
	A(S,S) = \sum_{T\in \binom{[n]}{k}} B^*(S,T)B(T,S)= \sum_{T\in \binom{[n]}{k}} \overline{B(T,S)}B(T,S).\]
	Since all of the terms in the first sum are nonnegative, the tropicalization ($\nu$ values) of the sum is the tropical sum (maximum) of the tropicalization of the summands. Moreover, for any $a\in \cK$, 
 $\nu(a) = \nu(\overline{a})$, so $\nu(\overline{a}a) = 2\nu(a)$. All together this gives
	\begin{equation}\label{eq:minorsPD}
		\nu(A(S,S)) = \max_{T\in \binom{[n]}{k}} 2\nu(B(T,S)).
	\end{equation}    
    Let $U\in \kk^{n \times n}$ be generic and let $\tilde{B} = UB$. Then it follows from \eqref{eq:minorsPD}  and Lemma~\ref{lem:genCoordChange} that
    \[\val(A(S,S))  = 2\val(\tilde{B}(T, S))\]
for any $T \subset [n]$ with $|T| = |S|$.  The values $\nu(\tilde{B}(T, S))$ are tropicalizations of $n \times n$ minors of the $n \times 2n$ matrix $(\tilde{B}~I)$, and they lie in the linear subspace $\Ln$.  Thus
\[
\frac{1}{2}\val(A(S,S)) \in \trop({\rm Gr}_{\cK}(n, 2n)) \cap \Ln.
\]
However $\trop({\rm Gr}_{\cK}(n, 2n))$ and $\Ln$ are invariant under scaling, so we also have 
\[
\val(A(S,S)) \in \trop({\rm Gr}_{\cK}(n, 2n)) \cap \Ln.
\]

    ($\supseteq$)
    Suppose $M \in \cK^{n\times 2n}$ is such that the $\nu$ values of the $n\times n$ minors of $M$ belong to $\Ln$.  
    We may assume that the valuation of the minor $M([n],\{n+1, \hdots, 2n\})$ is zero, and that $M$ has the form $\begin{pmatrix}B & I\end{pmatrix}$ for some $n\times n$ matrix $B$.
    Then the valuation of the minors $B(T,S)$ are independent of $T$. 
 Consider the positive definite matrix $A = B^TB$.  Then for any $S,T \subset [n]$ with $|S|=|T|=k$,
    by \eqref{eq:minorsPD} we have 
    $$ \nu(B(T,S)) =  \frac{1}{2}\max_{T'\in \binom{[n]}{k}} 2\nu(B(T',S)) = \frac{1}{2}\nu(A(S,S)),$$
    completing the proof.
\end{proof}

We remark on the dimension of the polyhedral complexes in Theorem~\ref{thm:TropGrassFlagGrassPD}. 

\begin{proposition}\label{prop:dim}
 If $\cK$ is algebraically closed or real closed, the polyhedral complex in Equation \eqref{eq:mainset} of Theorem~\ref{thm:TropGrassFlagGrassPD} has dimension $\binom{n+1}{2}+1$. Consequently, for $\cK = \cR$ or $\cC$, the polyhedral complex $\trop(\PM_n^+(\cK))$ has dimension $\binom{n+1}{2}$.
\end{proposition}
\begin{proof}
The affine tropical flag variety  $\trop({\rm Fl}_{\cK}(n)) \subset \prod_{k=0}^n \R^{\binom{n}{k}}$ is invariant under the tropical scaling  of each factor as in Equation \eqref{eq:tropScaling}.
Therefore, by Lemma~\ref{lem:submodular}, any polyhedral cone in the affine tropical flag variety has the same dimension as its intersection with the submodular cone. Thus it suffices to show that $\trop({\rm Fl}_{\cK}(n))$ has dimension $\binom{n+1}{2}+1$.

The flag variety has dimension $\binom{n}{2}$, as a dense subset of it can be parametrized by upper triangular matrices with ones on the diagonal. Thus the affine flag variety ${\rm Fl}_{\cK}(n)$ has dimension $\binom{n+1}{2} +1 = \binom{n}{2} + n+1$, accounting for the $n+1$ extra dimensions from scaling.  
If $\cK$ is algebraically closed, this implies that the tropicalization  $\trop({\rm Fl}_{\cK}(n))$ is a pure polyhedral complex of the same dimension $\binom{n+1}{2}+1$ by the Structure Theorem of Tropical Geometry. 

If $\cK$ is not algebraically closed, it is a priori possible that $\trop({\rm Fl}_{\cK}(n))$ is a (not necessarily pure) polyhedral complex of smaller dimension. 
However, when $\cK$ is real closed, it was shown in \cite[Proposition $4.10^{\rm trop}$]{Boretsky} that the tropicalization of the positive part ${\rm Fl}_{\cK}^{> 0}(n)$ of the flag variety, which is contained in $\trop({\rm Fl}_{\cK}(n))$, has dimension $\binom{n+1}{2}+1$, by showing that the Marsh-Rietsch (bijective) parameterization $\cK_{>0}^{\binom{n+1}{2}+1} \to {\rm Fl}_{\cK}^{> 0}(n)$ tropicalizes to a (bijective) parametrization $\R_{>0}^{\binom{n+1}{2}+1} \to \trop({\rm Fl}_{\cK}^{> 0}(n))$.
In \cite{Boretsky}, the tropical flag variety is considered projectively, so the dimensions used there do not account for the $n+1$ extra dimensions from scaling. The map easily extends to the affine setting. 

For the last assertion, we note that the set $\cone(\PM_n^+(\cK))$ is obtained from $\PM_n^+(\cK)$ by global scaling. Its dimension and the dimension of its tropicalization is therefore one more than the respective dimensions for $\PM_n^+(\cK)$.
\end{proof}

%\begin{corollary}
%For $\cK = \cR$ or $\cC$, the set $\trop(\PM_n^+(\cK))$ has dimension $\binom{n+1}{2}$.
%\end{corollary}
%\begin{proof}
%The set $\cone(\PM_n^+(\cK))$ is obtained from $\PM_n^+(\cK)$ by global scaling. Its dimension and the dimension of its tropicalization is therefore one more than the respective dimensions for $\PM_n^+(\cK)$. The result then follows from Proposition~\ref{prop:dim}.
%\end{proof}

We now explore some other consequences of the description given in Theorem~\ref{thm:TropGrassFlagGrassPD}.

\begin{corollary}\label{cor:equalitydependingonn}
The set of tropicalized principal minors of $n\times n$ positive definite matrices is contained in the set of $M^\natural$-concave functions on $2^{[n]}$ with value zero on $\varnothing$. 
This containment is an equality for $n\leq 5$, but it is strict for $n \geq 6$.
\end{corollary}

\begin{proof}
By Theorem~\ref{thm:TropGrassFlagGrassPD}, the set of tropicalized principal minors of $n\times n$ positive definite matrices coincides with the tropical flag variety inside the submodular cone.  The $M^\natural$-concave functions coincide with the flag Dressian by Proposition~\ref{prop:Mconcave}.  By Section 5 of~\cite{BEZ}, the tropical flag variety is contained in the flag Dressian; this containment is an equality for $n\leq 5$ and is strict for $n \geq 6$. This is still the case after intersecting with the submodular cone by Lemma \ref{lem:submodular}.
\end{proof}

If $w \in \R^{2^{[n]}}$, for each $k\in [n]$ let $w_k$ denote
the restriction of $w$ to subsets of size $k$. 
As discussed in Section~\ref{sec:FlagVariety}, the fact that $\nu(\PM^+_n(\cK)) \subset \nu({\rm Fl}_{\cK}(n))$ implies the following corollary.

\begin{corollary}
For $w\in \nu(\PM^+_n(\cK))$, 
the homogenization of the restriction $(w_k, w_{k+1})$ belongs to 
$\nu(\Gr_{\cK}(k+1,n+1))$ for any $0 \leq k \leq n-1$. 
\end{corollary}

Recall from Lemma~\ref{lem:layers} that if a function $w : 2^{[n]} \rightarrow \R$ is $M^\natural$-concave and strictly submodular, then each cell in the regular subdivision of $\{0,1\}^n$ induced by $w$ via upper hull is contained in a layer $\{x\in [0,1]^n : k\leq \sum_i x_i \leq k+1\}$ for some $k$.
Even when $w \in \nu(\PM^+_n(\cK))$ is not strictly submodular, we can perturb $w$ and get a subdivision where every cell is contained in one of these layers.
 
\begin{theorem}\label{thm:repMatroid}
 The regular subdivision of $[0,1]^n$ induced by the tropicalization of the principal minors of any positive definite matrix is a coarsening of a subdivision of $[0,1]^n$ into dehomogenized realizable matroid polytopes. Specifically, for points in $\nu(\PM^+_n(\cR))$ or $\nu(\PM^+_n(\cC))$, the matroids are realizable over $\R$ or $\C$, respectively. 
\end{theorem}

\begin{proof}
 Let $w\in \nu(\PM^+_n(\cK))$ for $\cK = \cR$ or $\cC$. By Theorem~\ref{thm:TropGrassFlagGrassPD}, $w$ belongs to  $\nu({\rm Fl}_{\cK}(n)) \cap {\rm SUBMOD}_n$.
  By Lemma~\ref{lem:submodular}, for arbitrarily small $\lambda\in \Gamma^{n+1}$, the point $\lambda\cdot w$ is strictly submodular. It also belongs to $\nu({\rm Fl}_{\cK}(n))$ and, in particular, is $M^{\natural}$-concave. 
  By Lemma~\ref{lem:layers}, each cell of the subdivision induced by $\lambda\cdot w$ is a subset of $\{x\in [0,1]^n: k\leq \sum_i x_i \leq k+1\}$ for some $k$. Because $\lambda$ was taken arbitrarily small, the subdivision induced by $\lambda\cdot w$ is a refinement of that induced by $w$. We conclude by checking that each cell induced by  $\lambda\cdot w$ is a  dehomogenized realizable matroid polytope. 

  Let $w' =(w'_k, w'_{k+1})$ denote the restriction of $\lambda\cdot w$ to subsets of size $k$ and $k+1$. Since $\lambda\cdot w\in \nu({\rm Fl}_{\cK}(n))$, the homogenization of $w'$ is an element of 
   $\nu(\Gr_{\cK}(k+1,n+1))$ as described in Section~\ref{sec:FlagVariety}. In particular, any cell in the subdivision of $\{x\in [0,1]^{n+1}: \sum_{i}x_i = k+1\}$ induced by the homogenization of $w'$ is the matroid polytope of a matroid realizable over the residue field of $\cK$. For $\cK = \cR$ or $\cC$, it follows that the matroid is realizable over $\R$ or $\C$, respectively.  Dehomogenizing, i.e.\ dropping the coordinate $x_{n+1}$, gives the result.
\end{proof}

\section{Positive semidefinite matrices} 
\label{sec:PSD}

In this section, we will replace positive definite (PD) matrices with positive semidefinite (PSD) matrices, and describe the tropicalization of their principal minors. The set of principal minors of PSD matrices live in $\cR_{\geq 0}$, and they form the closure of the set of principal minors of PD matrices.  The tropicalization of a subset of $\cR_{\geq 0}^n$ is  defined as the closure of the image of the (negative) valuation map from $\nu : \cR^n \rightarrow (\R \cup \{-\infty\})^n$.  

%As before, the tropicalization of a subset of $\R_{\geq 0}^n$ or $\cR_{\geq 0}^n$ can be defined as the logarithmic limit set (where the logarithm of $0$ is taken to be $-\infty$) or the closure of the image of the (negative) valuation map from $\nu : \cR^n \rightarrow (\R \cup \{-\infty\})^n$.  They coincide for semialgebraic sets.

\subsection{Tropicalization and Taking Closures}

Let $\cR$ be a real field with a nontrivial nonarchimedean valuation $\nu$ compatible with the ordering.  Assume for simplicity that the value group is $\R$, as one can take real closed field extensions without affecting the tropicalization. For example, $\cR$ can be the field of Hahn series with real coefficients.  

The field $\cR$ has the {\em order topology} given by a basic open sets of open intervals $(a,b) = \{r \in \cR \mid a < r < b\} \subset \cR$.  On the other hand it has a metric space structure where the distance between $r,s\in \R$ is $|r-s| = \exp(\nu(r-s))$.  When the nonarchimedean valuation $\nu$ is nontrivial, these two topologies coincide.  The space $\cR^n$ can be endowed with the product topology whose basic open sets are products of open intervals.  This topology coincides with the one induced by the $p$-norm $|\!|x|\!|_p = (|x_1|^p+\cdots + |x_n|^p)^{\frac{1}{p}}$ for any $p \in [1, \infty]$.  
Although the order topology of real Puiseux series $\cR$ is totally disconnected, since it is equivalent to the metric topology given by the valuation, we can use convergent sequences to define closures in $\cR^n$.

On the tropical side, we can map $\R \cup \{-\infty\}$ bijectively to $\R_{\geq 0}$ via the exponentiation map and endow it with the usual Euclidean topology on $\R_{\geq 0}$.

\begin{lemma}\label{lem:ValCont}
    The valuation map $\nu: \cR \rightarrow \R \cup \{-\infty\}$ is continuous.
\end{lemma}

\begin{proof}
Since $\cR$ is a metric space, it suffices to prove that $\nu$ is sequentially continuous.  Let $(a^i)_{i \in \N}$ be a sequence of points in $\cR$ converging to $c \in \cR$.  We wish to show that $(\nu(a^i))$ converges to $\nu(c)$. 
If $c = 0$, then $|a^i| \rightarrow 0$, so $\nu(a^i)\rightarrow -\infty$. If $c \neq 0$, then there exists an $N \in \mathbb{N}$ such that $|a^i - c| < |c|$ for all $i \geq N$.  By the nonarchimedean triangle inequality, this implies that $\nu(a^i) = \nu(c)$ for all $i \geq N$, so the conclusion holds.  
\end{proof}

For semialgebraic sets, taking the extension is compatible with taking closures. If $S$ is a semialgebraic set defined by some formula $\varphi$, then its closure is defined as the set of $x$ satisfying the formula 
\[
    \textstyle \forall \varepsilon (\varepsilon > 0 \implies \exists y (\varphi(y) \wedge \sum(x_i - y_i)^2 < \varepsilon)).
    \]
As this expression is valid over every real closed extension $\cR \supset \R$ we have \[
(\overline{S})_\cR = \overline{(S_\cR)}.
\]

\begin{lemma}
\label{lem:tropclosure}
    For any semialgebraic subset $S \subset \cR_{>0}^n$, $\overline{\nu(S)} = \nu(\overline{S})$.
\end{lemma}

\begin{proof}
By \cite[Theorem 6.9]{JSY} (last statement), $\nu(\overline{S})$ is closed in $(\R \cup \{-\infty\})^n$, so the inclusion $\overline{\nu(S)} \subseteq \nu(\overline{S})$ follows from this.  (Note that they use the exponentiated tropicalization in \cite{JSY}.)
The other inclusion follows from the continuity of the valuation map shown in Lemma~\ref{lem:ValCont}.
%suppose $(s^i)_{i \in \N}$ is a sequence of points in $S$ converging to $x \in \cR_{\geq 0}^n$. We wish to show that $\nu(x) \in \overline{\nu(S)}$. As shown in the proof of \cite[Lemma 6.4]{JSY}, if $x_k \neq 0$, then $\nu(s_k^i) = \nu(x_k)$ for all sufficiently large values of $i$.  If $x_k = 0$, then we must have $\nu(s_k^i) \rightarrow -\infty$.  Otherwise, if $\nu(s_k^i) \geq a$ for some $a \in \R$, then any positive element of valuation $a$ is a lower bound for $s_k^i$, contradicting the assumption that $s_k^i \rightarrow x_k = 0$.  It follows that  $\nu(s^i) \rightarrow \nu(x)$, so $\nu(x) \in \overline{\nu(S)}$ as desired.
\end{proof}

\begin{remark}
For non-semialgebraic sets in $\R_+^n$, it is not always the case that logarithmic limit sets are compatible with taking closures.  
%$\overline{\trop_{>0}(S)} = \trop_{\geq 0}(\overline S)$.  
For example, let $S = \{(s,s^{-s}): s\in \R_+\}$. Setting logarithm of $0$ to be $-\infty$, we have 
 \[\lim_{t \rightarrow \infty} \log_t \overline{S} = \{(a,0):a\leq 0\}\cup\{(0,a):a\leq 0\} \cup\{(a,-\infty):a\geq 0\} \subset (\R \cup \{-\infty\})^2.\]  The last of the three  pieces is not contained in $\overline{\lim_{t\rightarrow \infty} \log_t S}$.  The topology on $(\R \cup \{-\infty\})^2$ is derived from $\R_{\geq 0}^2$ via exponentiation $(x,y)\to (e^x,e^y)$.

\end{remark}

\subsection{Tropicalizing Principal Minors of PSD matrices}

The following theorem is the analogue of Theorem~\ref{thm:TropGrassFlagGrassPD} for PSD matrices.  As before, let $\cR$ and $\cC$ be real closed and algebraically closed fields respectively, with nontrivial nonarchimedean valuation.  Let $\PM^{\succeq 0}_n(\cK) \subset \cK^{2^n}$ be the set of principal minors of $n \times n$ symmetric or Hermitian positive semidefinite matrices over $\cK$.

\begin{theorem}\label{thm:TropGrassFlagGrassPSD} 
For any positive integer $n$ and a field $\cK = \cR$ or $\cC$, the following subsets of $(\R \cup \{-\infty\})^{2^n}$ coincide: 
	\[
    \trop({\rm cone}(\PM^{\succeq 0}_n(\cK))
    \ \ = \ \ 
    \overline{\trop({\rm Gr}_{\cK}(n, 2n)) \cap \Ln}
    \ \  = \ \
    \overline{{\rm trop}({\rm Fl}_{\cK}(n)) \cap {\rm SUBMOD}_n}
    \]
    where the closure is taken in $(\R \cup \{-\infty\})^{2^n}$.
\end{theorem}

\begin{proof}
The cone of PSD matrices is the closure of the cone of PD matrices, and the principal minor map is a polynomial map, so we have $PM^{\succeq 0}_n = \overline{PM^+_n}$.  The statement of the theorem then follows from Lemma~\ref{lem:tropclosure}.
\end{proof}

We do not know whether 
\begin{equation}
    \label{eqn:closures}
\begin{aligned}
    \overline{\trop({\rm Gr}_{\cK}(n, 2n)) \cap \Ln} & = \overline{\trop({\rm Gr}_{\cK}(n, 2n))} \cap \overline{\Ln}, \text{~~or}\\
\overline{{\rm trop}({\rm Fl}_{\cK}(n)) \cap {\rm SUBMOD}_n} & = \overline{{\rm trop}({\rm Fl}_{\cK}(n))} \cap \overline{{\rm SUBMOD}_n}.
\end{aligned}
\end{equation}

In general, taking intersections does not commute with taking closure, in $(\R\cup \{-\infty\})^{n}$, of polyhedra in $\R^n$.  For example, for the two half-spaces $H_1$ and  $H_2$ defined by $y\leq 2x$ and $2y \geq x$ respectively, the point $(-\infty, -\infty)$ is in $\overline{H_1}\cap\overline{H_2}$ but not in $\overline{H_1\cap H_2}$.  

Even in the context of $M$-concave functions, taking closure does not behave well.  There exists an example of two matroid polytopes $Q \subset P$ where $Q$ never appears as a face in a matroidal subdivision of $P$.  Concretely, we can take $P$ to be the matroid polytope of the ternary projective plane and  $Q$ to be one of the square-based pyramids inside an octahedral face of $P$. Proposition~32 of \cite{OPS} says that $P$ does not have a nontrivial matroidal subdivisions.
This means that the  function taking value $0$ on vertices of $Q$ and $-\infty$ on remaining vertices of $P$ is an $M$-concave function defined on vertices of $P$, but it is not in the limit of the $M$-concave functions with finite values on vertices of $Q$.

One can attempt to prove directly that the tropicalization of principal minors of positive semidefinite matrices coincide with the sets on right hand sides of \eqref{eqn:closures}.  However the technical Lemma~\ref{lem:TechnicalLemma} does not hold if we allow $-\infty$.   See Conjecture~\ref{conj:technical}.

\section{Inequalities from \texorpdfstring{$M^\natural$}--concavity}\label{sec:inequalities}
By the Lifting Lemma \cite{JSY}, for any semi-algebraic set $S$, any tropical polynomial inequality valid on $\trop(S)$ can be lifted to a usual polynomial inequality valid on $S$.  
In this section we use the theory of Lorentzian polynomials \cite{BH,AGV} to derive new inequalities on principal minors on positive definite semidefinite matrices that lift the $M^\natural$-concavity inequalities satisfied by the tropicalization, based on the fact that the polynomial $f = \det({\rm diag}(x_1,\hdots,x_n)+A)$ is Lorentzian for any positive definite matrix $A$. Furthermore, we provide examples to show that these inequalities are tight. 

For $D=\{d_1,\hdots,d_k\} \subset [n]$, let $\partial^Df$ denote the partial derivatives of $f$ with respect to the variables $x_{d_1},\hdots,x_{d_k}$. For $S\subset [n]$, let $A_S$ denote the principal minor indexed on the rows and columns by $S$.
To simplify notation, we write $Sij = S \cup \{i,j\}$. We first derive inequalities on the coefficients of arbitrary Lorentzian polynomials. 

\begin{theorem}\label{thm:inequ1}
 	Let $f = \sum_{S\subseteq [n]}c_S\,{\bf x}^{[n]\backslash S}y^{|S|}$ be a Lorentzian polynomial in the variables $x_1, \dots, x_n ,y$. If $n\geq 4$, then for any $S \subseteq [n]\setminus \{1,2,3,4\}$ and any $r\in \R$, the coefficients of $f$ satisfy
	\[ (r+1) \, c_{S14} c_{S23} +r(r+1) \, c_{S13} c_{S24} \geq r \, c_{S12} c_{S34}\]
 and all analogous inequalities obtained by permutations of indices. 
Similarly, if $n\geq 3$, then for any $S \subseteq [n]\setminus \{1,2,3\}$ and any $r\in \R$, the coefficients satisfy
	\begin{equation*}\label{equ:PMIneq}
	    (r+1) \, c_{S1} c_{S23} + r (r+1) \, c_{S2} c_{S13} \geq r \, c_{S3} c_{S12}
	\end{equation*} 
  and all analogous inequalities obtained by permutations of indices.
  
  Moreover, both of these inequalities are tight for the subset of Lorentzian polynomials whose coefficients $c_S$ are the principal minors $A_S$ of a positive semidefinite matrix $A$.
\end{theorem}

We will make use of the following lemma.

\begin{lemma}\label{lem:quadLorentzian}
	If $q = c_{34}x_1x_2 + c_{24}x_1x_3+c_{23}x_1x_4+c_{14}x_2x_3+c_{13}x_2x_4 + c_{12}x_3x_4$ is Lorentzian, then for any $r\in \R$
	\begin{equation*}%\label{equ:PMIneq}
	    (r+1)\,  c_{14} c_{23} + r (r+1) \,c_{13} c_{24} \geq r \, c_{12} c_{34}.
	\end{equation*} 
\end{lemma}
\begin{proof}
    By the definition of Lorentzian polynomials, the coefficients $c_{ij}$ are all nonnegative and the symmetric matrix 
    \[
   Q = \nabla^2q= \begin{pmatrix}
        0 & c_{34} & c_{24} & c_{23} \\ c_{34} & 0 & c_{14} & c_{13} \\ c_{24} & c_{14} & 0 & c_{12} \\ c_{23} & c_{13} & c_{12} & 0
    \end{pmatrix}
    \]
    representing the quadratic form $q$ has at most one positive eigenvalue. 
    %Note that since the entries of $Q$ are all nonnegative, it also must have at least one nonnegative eigenvalue. 
    It follows that $\det(Q)\leq 0$. 
    
    The determinant of $Q$ is equal to the discriminant of the quadratic polynomial
    \begin{equation}     
    \label{eqn:quadratic}
    c_{13} c_{24}\, r^2 + ( c_{13} c_{24} + c_{14} c_{23} - c_{12} c_{34} ) \, r + c_{14} c_{23}
     =     (r+1)\, c_{14} c_{23} + r(r+1) \, c_{13} c_{24} - r \, c_{12} c_{34}
        \end{equation}
    with respect to $r$. 
   
    Since the extremal coefficients $c_{13} c_{24}$ and $c_{14} c_{23}$ and nonnegative and the discriminant is $\det(Q)\leq 0$, it follows that the quadratic polynomial \eqref{eqn:quadratic} is nonnegative for all $r\in \R$.  
\end{proof}

\begin{proof}[Proof of Theorem~\ref{thm:inequ1}]
We first show that the first inequality implies the second. 
If $f $ is Lorentzian, then so is its polarization \[{\rm Pol}(f) = \sum_{S\subseteq [n]}\frac{c_S}{\binom{n}{|S|}}{\bf x}^{[n]\backslash S}e_{|S|}(y_1, \hdots, y_n),\]
where $e_{j}(y_1, \hdots, y_n)$ is the $j$th elementary symmetric polynomial in $y_1, \hdots, y_n$; see \cite[Prop.~3.1]{BH}.
For any set $S \subseteq [n]\setminus \{1,2,3\}$ with $|S|=k-1 \leq n-3$ and $i,j\in \{1,2,3\}$, the coefficient of ${\bf x}^{[n]\backslash Sij}{\bf y}^{[k-1]}$ in ${\rm Pol}(f)$ is $\binom{n}{k-1}^{-1}c_{Sij}$  and the coefficient of ${\bf x}^{[n]\backslash Si}{\bf y}^{[k]}$ in ${\rm Pol}(f)$ is $\binom{n}{k}^{-1}c_{Si}$. 
Assuming the first inequality holds for all Lorentzian polynomials and applying it to ${\rm Pol}(f)$ with the variables $(x_1,x_2,x_3,y_k)$ in place of 
$(x_1,x_2,x_3,x_4)$ and  
$\{x_i:i\in S\}\cup \{y_1, \hdots, y_{k-1}\}$ in place of 
$\{x_i:i\in S\}$  gives the desired inequality.

Now we show the first inequality. 
For general $D\subseteq [n]$, 
$\partial^Df = \sum_{S}c_{S}{\bf x}^{[n]\backslash (S\cup D)y^{|S|}}$ where the sum is taken over $S\subseteq [n]$ for which $S\cap D=\emptyset$. 
Now we fix $S\subseteq[n]\backslash \{1,2,3,4\}$ and let $D = [n]\backslash (S\cup \{1,2,3,4\})$. 
Consider the polynomial 
$g$ obtained by specializing the derivative $\partial^{|S|+2}_{y}\partial^Df$ to $y=0$ and $x_i=0$ for $i\in S$:
	\[
	g = \left(\partial^{|S|+2}_{y}\partial^Df\right)|_{\{y=0,\  x_i=0 \ \forall i\in S\}}.
	\]
Taking derivatives and specializing variables to zero 
preserves the Lorentzian property by \cite[Theorem~2.30]{BH}, so $g$ is Lorentzian. 
One can check that $g$ is given by
\[    g = (|S|+2)!(c_{34S}x_1x_2 + c_{24S}x_1x_3+c_{23S}x_1x_4+c_{14S}x_2x_3+c_{13S}x_2x_4 + c_{12S}x_3x_4).
\] 
The desired inequality then holds by Lemma~\ref{lem:quadLorentzian}.

Finally, in Examples~\ref{ex:Ar1} and \ref{ex:Ar2} we will see that these inequalities are tight for the minors of positive semidefinite matrices when $n=4$ and $3$, respectively, and $S=\emptyset$. Appending identity matrices of arbitrary size shows that these inequalities are tight for general $n$ and $S$. 
\end{proof}

\begin{corollary}
	Let $A$ be an $n\times n$ positive semidefinite matrix. If $n\geq 4$, then for any $S\subseteq[n]\backslash \{1,2,3, 4\}$ and any $r\in \R$,
 	\[
	(r+1) \, A_{S14} A_{S23} + r(r+1)\, A_{S13} A_{S24} \geq r\, A_{S12} A_{S34}
	\]
 Similarly, if $n\geq 3$, then for any $S\subseteq[n]\backslash \{1,2, 3\}$ and any $r\in \R$,
  	\[
	(r+1) \, A_{S1} A_{S23} + r(r+1)\, A_{S2}A_{S13}  \geq r\, A_{S3}A_{S12}.
	\]
 Moreover these inequalities are tight for any $r\in \R$.
\end{corollary}

\begin{proof}
    The polynomial $f = \det({\rm diag}(x_1,\hdots,x_n)+yA) = \sum_{S\subseteq[n]}A_S{\bf x}^{[n]\backslash S}y^{|S|}$ is stable \cite{BB06} and hence Lorentzian \cite{BH}. The result then follows from Theorem~\ref{thm:inequ1}.
\end{proof}

\begin{example}\label{ex:Ar1}
	For $r\in \R$, consider the matrix 
	\[A=\left(
	\begin{array}{cccc}
 1 & 1 & 1 & r \\
 1 & 2 & -r+2 & 2 r-1 \\
 1 & -r+2 & r^2-2 r+2 & -r^2+3 r-1 \\
 r & 2 r-1 & -r^2+3 r-1 & 2 r^2-2 r+1 
	\end{array}
	\right).
	\]
One can check that $A$ is a positive semidefinite matrix of rank two. For example, all of the principal minors are sums of squares in $r$. Moreover, the $2\times 2$ minors of $A$ satisfy
	\[(r+1) A_{14} A_{23} + r (r+1) A_{13} A_{24} = r A_{12} A_{34},
	\]
showing the inequalities are tight.
\end{example}

\begin{example}\label{ex:Ar2}
    Interestingly, there is not a parametrized family of $3\times 3$ positive semidefinite matrices for which the second inequality of Theorem~\ref{thm:inequ1} holds with equality. The image of the set of $3\times 3$ positive semidefinite matrices under the map $A\mapsto (A_{1} A_{23}, A_{2}A_{13}, A_{3}A_{12})$ is not closed. We see that the equality is only attained in the limit. 

    Fix $r\in \R$ and consider the following parameterized collection of  $3\times 3$ matrices:
    \[
    A= \begin{pmatrix}
 1+\varepsilon  & \sqrt{1-\lambda(r+1)^2} & \sqrt{1-\lambda} \\
 \sqrt{1-\lambda(r+1)^2} & 1+\varepsilon & \sqrt{1-\lambda r^2 } \\
 \sqrt{1-\lambda} & \sqrt{1-\lambda r^2} & 1+\varepsilon
    \end{pmatrix}.
    \]
    For $0<\lambda<\min\{1/(r+1)^2, 1/r^2, 1\}$ and $\varepsilon>0$, this matrix has real entries and positive $1\times 1$ and $2\times 2$ principal minors. Moreover, for fixed $\varepsilon>0$, the limit of $\det(A)$ as $\lambda\to 0$ is $\varepsilon^2 (3 + \varepsilon)$. Thus for sufficiently small $\lambda>0$, the matrix $A$ is positive definite. 
    One can compute that 
\[	(r+1) \, A_{1} A_{23} + r(r+1)\, A_{2}A_{13}  - r\, A_{3}A_{12} = 
(1 + r + r^2) \varepsilon (1 + \varepsilon) (2 + \varepsilon). \]
In particular, as $\varepsilon\to 0$, the limit is zero and the desired inequality becomes tight.
\end{example}

\begin{remark}
One can check that the linear inequalities in Theorem~\ref{thm:inequ1} cut out a quadratic cone. 
More precisely, a point $(x,y,z)\in \R^3$ satisfies the inequality $(r+1)x + r(r+1)y\geq r z$ for all $r\in \R$ if and only if $2(xy+xz+yz)\geq x^2+y^2+z^2$ and $x+y+z\geq 0$. 
\end{remark}

\section{Proof of the technical lemma}\label{sec:proof}
This section is dedicated to the proof of Lemma~\ref{lem:TechnicalLemma}, which was used in the proof of Theorem~\ref{thm:TropGrassFlagGrass_AnyField}.
Throughout this section we take $B\in \cK^{n\times n}$ to be an upper triangular matrix
whose upper justified minors $\det(B(\{1,\hdots, |S|\}, S))$ are all nonzero.
We consider the function 
\begin{equation}\label{eq:wMinors}
w(T,S) = \nu(\det(B(T,S)))
\end{equation}
which is defined on pairs $(T,S)$ of subsets of the same size and takes values in $\R\cup \{-\infty\}$.

Before the proof, we introduce some notation. 
We define the following partial order on $\binom{[n]}{k}$. Given $S, T \in \binom{[n]}{k}$, we say that 
$$S\preceq T \quad \text{ if } \quad \left|\{s\in S : s\leq t_j\}\right| \geq j \text{ for all } j=1, \hdots, k$$
where  $t_1<t_2<\hdots<t_k$ are the elements of $T$. 
In particular, since $B$ is upper triangular, $\det(B(T,S))=0$ whenever $S\prec T$.
We call $I\subset [n]$ an \emph{interval} if it is a collection of consecutive numbers.  
That is, whenever $i,k\in I$ and $i<j<k$ we have $j\in I$. 
We use $[a]$ denote the interval $\{1,\hdots, a\}$.

\begin{proposition} \label{prop: properties}
The function $w$  in \eqref{eq:wMinors} satisfies the following: 
\begin{itemize}
\item[(i)] $w(\{1,..,|S|\}, S) \neq -\infty$ for all $S$.
\item[(ii)] $w(T,S) = -\infty$ whenever $S\prec T$. 
\item[(iii)] for any interval $I$ with $\max(I) \leq \min(T\cup S)$,
 $w(I\cup T, I \cup S) = \sum_{i\in I} w(i,i) + w(T,S)$.
\item[(iv)] if $\max(S\cup T) \leq t$ and $\max(S\cup T) \leq s$, then  
$w(Tt,Ss) = w(T,S) + w(t,s)$.
\item[(v)] for any $S$, the function $T \mapsto w(T,S)$ satisfies the 3-term tropical Pl\"ucker relations. 
\end{itemize}
\end{proposition}
\begin{proof}
Item (i) holds by assumption. Whenever $S\prec T$, $\det(B(T,S))=0$ and $w(T,S) = -\infty$, since $B$ is upper triangular, showing (ii). For (iii),  the conditions on $I,S,T$ imply that $\det(B(I\cup T, I \cup S)) = \left(\prod_{i\in I} b_{ii}\right) \cdot \det(B(T, S))).$
Taking valuations of both sides proves the claim. 
Under the assumptions of (iv), $B(t,u)=0$ for all $u\in S$. If $t\leq s$, then the determinant of $B(Tt,  Ss)$ factors as product of $\det(B(T,S))$ and $B(t,s)$. Otherwise, both $\det(B(Tt,  Ss))$ and $B(t,s)$ will be zero.
Finally, for any fixed $S$, the 
values $w(T,S)$ are the maximal minors of an $n\times |S|$ matrix, which therefore satisfy the three-term Pl\"ucker relations. 
\end{proof}

To prove Lemma~\ref{lem:TechnicalLemma}, it therefore suffices to show the following. 

\begin{theorem}\label{thm:Techical}
 Let $w$ be any function 
defined on pairs $(T,S)$ of subsets of $[n]$ of the same size and taking values in $\R\cup \{-\infty\}$. If $w$ satisfies the properties (i)-(v) in Proposition~\ref{prop: properties} and  the function $F(S) =w(\{1,..,|S|\}, S)$ is submodular, 
 then 
\[w(\{1,\hdots, |S|\}, S) \geq w(T, S)\]
for all $S,T \subseteq [n]$ with $|S|=|T|$. 
\end{theorem}

For the remainder of the section, we assume $w$ is an arbitrary function satisfying 
these properties and for which $F(S) =w(\{1,..,|S|\}, S)$ is submodular. Our goal is to prove this theorem. We first prove it in the case $|S|=1$.

\begin{lemma}\label{lem:1x1caseB}
$w(j,k) \geq w(j+1,k)$ for all $1\leq j,k \leq n$.
\end{lemma}

\begin{proof}
If $j+1>k$, then $w(j+1,k)=-\infty$ by property (ii), and the inequality holds trivially. 
So we can assume that $j+1\leq k$. 
Let $I$ denote the interval $\{1,\hdots, j-1\}$. By submodularity, 
\begin{align*}
 w(Ij,Ij) + w(Ij,Ik) 
&= F(Ij) + F(Ik)\\
&\geq F(I) + F(I\cup \{j,k\})\\
&=w(I,I) + w( I\cup \{j,j+1\}, I\cup \{j,k\}).
\end{align*}
Using property (iv), this gives that 
\[
\sum_{i\in I} w(i,i) + w(j,j) + \sum_{i\in I} w(i,i) + w(j,k) \geq \sum_{i\in I} w(i,i) + \sum_{i\in I} w(i,i) + w(j,j) + w(j+1,k). 
\]
Canceling out diagonal terms gives $ w(j,k) \geq \ w(j+1,k)$. 
\end{proof}

For an interval $I $, we use $1+I$ to denote the shifted interval $\{1+i : i\in I\}$.

\begin{lemma}\label{lem:AddOneElemB}
Let $I$ be an interval and subset $S\subseteq [n]$ of the same size with $I \preceq S$. 
Then 
\[
w(I,S) - w(1+I, S) \geq w(I\backslash\!\max(I), S\backslash\!\max(S)) - 
w((1+I)\backslash\!\max(1+I), S\backslash\!\max(S)) \geq 0.
\]
\end{lemma}
\begin{proof}Let $s = \max(S)$, $i= \min(I)$, and $j = \max(I)$. Then $[i-1] \cup I = [j]$. 

First we assume that $i<\min(S)$, then $[i] \cap S = \phi$. For the left inequality, we use the submodularity inequality 
\[
F([i-1]\cup S) +
F([i]\cup S\backslash s) 
\geq 
F( [i-1]\cup S\backslash s) + F([i]\cup S).
\]
Using property (iv) and canceling our diagonal terms $\sum_{k=1}^{i-1}w(k,k) + \sum_{k=1}^{i}w(k,k)$, this inequality becomes
$$w(I,S) + w((I+1)\backslash(j+1)),S\backslash s)
\geq 
 w(I\backslash j,S\backslash s)+ w(1+I,S).
$$
The claim follows after rearranging terms. 

If $\min(S) = i$, then let $J$ be an interval such that $\min(J) = i$ and $J = I\cap S$. Let $I' = I \backslash J$ and $S' = S \backslash J$. Then, by the previous part, the claim holds for $I'$ and $S'$, that is
\[
w(I',S') - w(1+I',S') \geq w(I'\backslash\!\max(I'),S\backslash\!\max(S)) - w((1+I')\backslash\!\max(1+I'),S'\backslash\!\max(S')).
\]
By (iii), $w(I,S) = \sum_{k=i}^{\max(J)} w(k,k) + w(I',S')$ and by (iv) we have $w(1+I,S) = \sum_{k=i}^{\max(J)} w(k+1,k) + w(1+I',S')$. Thus, adding $\sum_{k=i}^{\max(J)} w(k,k) - w(k+1,k)$ to both sides of the above inequality proves the result.

The right inequality holds by induction on $|S|$. The base case $|I\backslash \max(I)|=1$ is covered by Lemma~\ref{lem:1x1caseB}. The inductive step is given by the left inequality. 
 \end{proof}

\begin{lemma}\label{k-j=1B}
Let $I$ be an interval and a subset $S\subseteq [n]$ of the same size. 
For any $t>\max(I)$ such that $(It\backslash\!\max(I)) \preceq S$, 
	\[w(I,S) \geq w(It\backslash\!\max(I),S)). \]
\end{lemma}
\begin{proof} 
Let $s = \max(S)$, $s' = \max(S\backslash\!s)$,  $i= \min(I)$, and $j = \max(I)$. 
	
($t=s$) By submodularity of $F$, we have 
\[
F([i-1]\cup S)+F([s']) \geq 
F([i-1]\cup S\backslash s) +F([s']s).
\]
Canceling diagonal terms $\sum_{k=1}^{i-1}w(k,k) + \sum_{k=1}^{s'}w(k,k)$ on both sides gives 
\[
w(I,S) \geq w(I\backslash j, S\backslash s) + w(s'+1, s).
\]
By Lemma~\ref{lem:1x1caseB}, $w(s'+1, s)\geq w(s,s)$. Then, using property (iv), we find
\[
w(I,S) \geq w(I\backslash j, S\backslash s) + w(s, s) = w((Is\backslash j), S).\]

($t< s$) We induct on $|S|$ and then $t-\min(I)$. The case $|S|=1$ is Lemma~\ref{lem:1x1caseB}. If $t-\min(I)=1$, then the inequality holds also by Lemma~\ref{lem:1x1caseB}. We therefore assume $|S|\geq 2$ and $t-\min(I)>1$. 

By property (v), the following maximum is attained at least twice
	\[
	\max\left\{ 
	w(Iab\backslash\{i,j\},S)) +  w(Ics\backslash\{i,j\},S)
	\right\},
	\]
	where the max is taken over all assignments $\{a,b,c\} = \{i,j, t\}$. 
	We break the argument into two cases, depending on which terms attain this maximum. 

		(Case 1) If the term  with $\{a,b\} = \{i,j\}$ attains the maximum then 
\[w(I,S) +  w(Its\backslash\{i,j\},S)\geq w(It\backslash j,S) +  w(Is\backslash i,S).
\]
Rearranging terms gives
	\begin{align*}
		w(I,S))  -  w(It\backslash j,S)
		&\geq w(Is\backslash i,S) - w(Its\backslash\{i,j\},S) \\
		& =w(I\backslash i),S\backslash s) - w(It\backslash\{i,j\},S\backslash s) \geq 0.
	\end{align*}
The last inequality follows by induction on $|S|$. 

		(Case 2) If the term  with $\{a,b\} = \{i,j\}$  does not attain the maximum, then the other two terms are equal. That is, 
			\[
	w(It\backslash j,S) + w(Is\backslash i,S)  = 
		w(Is\backslash j,S)  +  w(It\backslash i,S).	
	\]
Rearranging terms and assuming that $i<\min(S)$ gives		
\begin{align*}
w(It\backslash j,S) - w(It\backslash i,S)	& = 
		w(Is\backslash j,S) -  w(Is\backslash i,S) \\
		& = 
		w(I\backslash j,S\backslash s) -  w(I\backslash i,S\backslash s) \\
		& \leq w(I,S) -  w(1+I,S) \\
			& \leq w(I,S)  -  w(((1+I)t\backslash (j+1),S)\\
		& = w(I,S) - w(It\backslash i,S).
	\end{align*}
	The first inequality follows from Lemma~\ref{lem:AddOneElemB} and the fact that $(1+I)\backslash\!\max(1+I) = I\backslash i $, and the last inequality follows by induction on $t-\min(I)$, since $t -\min(1+I) < t -\min(I)$. We have $1+I \preceq S$ since $i<\min(S)$. 
	Simplifying the final inequality gives 
	\[w(It\backslash j,S) \leq w(I,S) ,\]
	as desired. 
 
 Now if $ i = \min(S)$, then let $J$ be an interval such that $\min(J) = i$ and $J = I\cap S$. Let $I' = I \backslash J$ and $S' = S \backslash J$. Then, by the previous part, the claim holds for $I'$ and $S'$, that is
\[
w(I't\backslash j ,S') \leq w(I',S').
\]
By (iii), $w(I,S) = \sum_{k=i}^{\max(J)} w(k,k) + w(I',S')$. Thus, adding $\sum_{k=i}^{\max(J)} w(k,k)$ to both sides of the above inequality finishes the proof.
\end{proof}

\begin{lemma}\label{lem:justifyInequality}
Let $S, T\subseteq [n]$ with $T\preceq S$.
Let $I$ be the interval with $\min(I)=\min(T)$ and $|I|=|S|=|T|$. 
Then $w(I,S) \geq w(T,S)$.  
\end{lemma}

\begin{proof} First we assume that $\max(S) \neq \max(T)$ and we proceed by induction on $|T|$. 
    For $|T|=1$, $I= T = \{\min(T)\}$ so the statement holds trivially.  Therefore we can assume $|S| = |T|\geq 2$.

Consider the first consecutive run of elements of $T$. 
That is, let $J$ be the largest interval with $\min(J) = \min(T)$ and $J\subseteq T$. 
We also induct on $|T|-|J|$.  If $|J|=|T|$, then $I = J = T$ and the inequality holds trivially. The case $|T|-|J|=1$ follows from Lemma~\ref{k-j=1B}.
	Therefore we may suppose $|T|-|J|\geq 2$.
	Let $j = \max(J)$. We will show by induction on $|T|$ that 
	\begin{equation}\label{eq:InductiveInequality}
	w(T,S) \leq \max_{t\in T\backslash J}  w((T(j+1)\backslash t, S).
	\end{equation}
	Since $(T(j+1)\backslash t)$ has a longer initial consecutive sequence 
	than $T$, it follows by induction on $|T|-|J|$ that 
	$w((T(j+1)\backslash t, S)\leq w(I,S)$, implying $w(T,S)\leq w(I,S)$, as desired.

	Let $s = \max(S)$. Suppose that $U = \{x,y\}$ attains the maximum 
	\[
	\max_U w(T(j+1)\backslash U,S\backslash s)
	\]
	taken over subsets $U\subseteq T\backslash J$ with $|U| = 2$.
By property (v), the maximum of 
\[
		\max_{\{a,b,c\} = \{j+1, x,y\}} w((Tab\backslash\{x,y\}) , S) + w((Tcs\backslash \{x,y\}) , S)
\]
	is attained at least twice. 
	In particular, after possibly relabeling $x,y$,  
	we have
	\[
	w(T, S) + w((Ts(j+1)\backslash \{x,y\}), S) \leq 
	w((T(j+1)\backslash y) , S) + w(Ts\backslash x, S).\]
 By property (iii), we can cancel the diagonal terms $w(s,s)$ and rearrange to get 
			\[
	 w((T(j+1)\backslash \{x,y\}) , S\backslash s)-w(T\backslash x, S\backslash s) \leq 
	w((T(j+1)\backslash y) , S)-w(T, S) .\]
	By induction on $|T|$, we have that 
	\[
	w(T\backslash x, S\backslash s)
	\leq 
	\max_t w((T(j+1)\backslash\{x,t\}), S\backslash s) 
	\]
	where $t$ runs over $(T\backslash J)\backslash x$. 
	By assumption on $\{x,y\}$, 
	this maximum is attained by $t=y$, giving that 
	\[
	w(T\backslash  x, S\backslash s) 
	\leq 
	w(T(j+1)\backslash\{x,y\}, S\backslash s).
	\]
	Together with the inequality from above, we get that 
$$
0 \leq w(T(j+1)\backslash \{x,y\}, S\backslash s) - w(T\backslash x , S\backslash s)) \leq 
w(T(j+1)\backslash y , S)) - w(T, S),$$
	showing that 
	\[w(T,S) \leq w(T(j+1)\backslash y , S))  \leq \max_{t\in T\backslash J}  w(T(j+1)\backslash  t, S).\]
 
 If $\max(S)=\max(T)=s$, let $J$ be the interval such that $\max(J) = \max(S) = \max(T)$, $J \subset T\cap S$ and it is of maximum length. We induct on the size of $J$. If $J = 1$, that is if $\max(S\backslash s)\neq \max(T\backslash s)$, then let $I' = I\backslash \!\max(I)$, $S' = S \backslash s$, and $T' = T \backslash s$. Then by the previous part $w(I',S')\geq w(T',S')$. Using Lemma~\ref{k-j=1B}, we get
 \begin{align*}
 w(I,S) &\geq w(Is\backslash\!\max(I) ,S) \\
        &= w(I'\!s,S)\\
        & = w(I',S') + w(s,s) \\
        &\geq w(T',S') + w(s,s)\\
        &= w(T,S).
 \end{align*}
  For the inductive step, we use similar argument as the one above and this finishes the proof. 
\end{proof}
	
Now we are ready to complete the proof of this section. 
	
\begin{proof}[Proof of Theorem~\ref{thm:Techical}] Let $S, T\subseteq [n]$. 
By Lemma~\ref{lem:justifyInequality}, 
$w(T,S) \leq w(I,S)$, where $I$ is the interval with $\min(I) = \min(T)$ and $|I|=|T|$. 
By Lemma~\ref{lem:AddOneElemB}, 
$ w(1+J, S)\leq w(J,S) $ for every interval $J$. 
Inducting then gives $w(k+J, S) \leq w(J,S) $ for arbitrary $k\in \Z_{>0}$. 
In particular, for $J = \{1,\hdots, |S|\}$ and $k=\min(T)-1$, this gives 
\[
w(T,S) \leq w(I,S) \leq w(\{1,\hdots, |S|\}, S),
\]
as desired.
\end{proof}	

\begin{remark} One might hope for inequalities of the form 
$w(T,S) \geq w(T',S)$ when $T\succeq T'$, but these do not hold in general. 
For example, consider the matrix 
	$$B = \begin{pmatrix}
		1 & 1 & 1 & 1\\
		0 & 1 & 2 & 3\\
		0 & 0 & 1 & 1 + t\\
		0 & 0 & 0 & 1\\
	\end{pmatrix}.$$
This matrix is upper triangular and all of its upper justified minors are nonzero
and have valuation $0$. Therefore the function $F$ is identically $0$ and thus submodular. 
However $w(\{1,3\},\{3,4\})) = -1$, whereas $w(\{2,3\}, \{3,4\})=0$.
\end{remark}

\begin{example}[$n=8, k=4$]
	To illustrate the strategy of the proof above, consider  $S = \{5,6,7,8\}$, and $T = \{1,4,5,7\}$. Here $J = \{1\}$. 
	Inequality \eqref{eq:InductiveInequality}  states that
	\[w(1457,S)\leq \max\{w(1257,S), w(1247,S), w(1245,S)\}.\]
	Let us show this inequality in the case that 
	$w(127,567)$ achieves the maximum among $\{w(124,567), w(125,567), w(127,567)\}$. In the proof above, this corresponds to $\{x,y\}= \{4,5\}$. 
	By the tropical Pl\"ucker relations, 
	the maximum
	\[\max_{\{a,b,c\} = \{2,4,5\}}\{w(17ab,S) + w(178c,S)\} \text{ is attained at least twice,}\]
	from which we can conclude that 
	\[
	w(1457,S) + w(1278,S) \leq 
	w(127a,S) + w(178b,S)
	\]
	for some assignment of $\{a,b\} = \{4,5\}$. Then $w(xyz8,S) = w(xyz,S')+w(8,8)$ for any subset $\{x,y,z\}\subset[7]$ where $S' = \{5,6,7\}$. 
	The inequality above give 
	\begin{align*}
		w(127a,S) - w(1457,S)
		& \geq w(1278,S) - w(178b,S)\\
		& =w(127,S') - w(17b,S').
	\end{align*}
	By induction, $w(17b,S')$ is bounded above by $\max\{w(127,S'),w(12b,S')\}$, which is bounded above by $w(127,S')$, by assumption. Therefore this difference is nonnegative, giving 
	\[
	w(1457,S) \leq w(127a,S) \leq \max_{x,y\in \{4,5,7\}}w(12xy,S).
	\]
	
	From this, we can continue by induction. 
	Using \eqref{eq:InductiveInequality} again gives 
	\begin{align*}
		w(1257,S) & \leq \max\{w(1237,S), w(1235,S)\},\\
		w(1247,S) & \leq \max\{w(1237,S), w(1234,S)\},\\
		w(1245,S)) & \leq \max\{w(1235,S), w(1234,S)\}.
	\end{align*}
	Finally, by Lemma~\ref{k-j=1B}, we see that these are all bounded above by  $w(1234,S)$. 
	That is, for $x,y,z \geq 4$
	\[
	\max_{x,y,z\geq 4} w(1xyz,S) \leq \max_{x,y\geq 4} w(12xy,S)  \leq \max_{x\geq 4} w(123y,S) \leq w(1234,S).
	\]	
\end{example}

It would be desirable to have analogues of Theorems~\ref{thm:TropGrassFlagGrass_AnyField} and~\ref{thm:TropGrassFlagGrassPD} for functions $F$ taking values in $\Gamma\cup\{-\infty\}$, rather than just $\Gamma$. 
However, the statement of Lemma~\ref{lem:TechnicalLemma}, which was needed in our proofs, would not be true if we just replaced $\R$ with $\R\cup \{-\infty\}$.  For example, matrices of the form $B = \begin{pmatrix}
    0 & a \\ 0 & b
\end{pmatrix}$ satisfy the submodularity condition but need not be top heavy.  The proof above would not work because we are not able to cancel out $-\infty$ from both sides of an inequality. To remedy this, we offer the following conjectural analogue of Lemma~\ref{lem:TechnicalLemma}, which should be enough to prove the analogues of Theorems~\ref{thm:TropGrassFlagGrass_AnyField} and~\ref{thm:TropGrassFlagGrassPD} with values in $\Gamma\cup\{-\infty\}$.

\begin{conjecture}    
\label{conj:technical}
    Let $B\in \cK^{n\times n}$ be an upper triangular matrix of rank $r$  whose top left $r\times r$ minor is nonzero and
    such that the function $2^{[n]} \rightarrow \R\cup\{-\infty\}$ given by  $S\mapsto \nu(B(\{1,\hdots, |S|\}, S))$ is submodular. Then 
\begin{equation}
		\nu(B(\{1,\hdots, |S|\}, S)) \geq \nu(B(T, S))
	\end{equation}
	for all $S,T \subset [n]$ with $|S|=|T|$. 
\end{conjecture}

\bibliography{bibfile}
\bibliographystyle{alpha}

\end{document}